\documentclass[a4paper,10pt,reqno]{amsart}
\usepackage{amsmath,amsfonts,amsthm,amssymb,color}
\usepackage[T1]{fontenc}
\usepackage{pdfsync}
\usepackage{csquotes}
\usepackage{graphicx}
\usepackage{pstricks}
\usepackage{lmodern}

\usepackage{tikz}

\newcommand{\be}{\beta}

\newcommand{\1}{{\bf 1}}

\newcommand{\rti}{\tilde{r}}
\newcommand{\uti}{\tilde{u}}

\newcommand{\etati}{\tilde{\eta}}
\newcommand{\lati}{\tilde{\la}}
\newcommand{\betati}{\tilde{\beta}}
\newcommand{\fouri}{\mathcal{F}}
\newcommand{\yti}{\tilde{y}}
\newcommand{\zti}{\tilde{z}}
\newcommand{\Hti}{\tilde{H}}

\newcommand{\etatiti}{\tilde{\tilde{\eta}}}

\newcommand{\R}{\mathbb R}

\newcommand{\cd}{\mathcal D}
\newcommand{\ce}{\mathcal E}
\newcommand{\cf}{\mathcal F}

\newcommand{\cj}{\mathcal J}

\newcommand{\cn}{\mathcal N}

\newcommand{\cs}{\mathcal S}
\newcommand{\cw}{\mathcal W}

\newcommand{\al}{\alpha}

\newcommand{\ga}{\gamma}
\newcommand{\gga}{\Gamma}

\newcommand{\la}{\lambda}

\newcommand{\si}{\sigma}

\newtheorem{theorem}{Theorem}[section]

\newtheorem{definition}[theorem]{Definition}

\newtheorem{lemma}[theorem]{Lemma}

\newtheorem{proposition}[theorem]{Proposition}

\theoremstyle{remark}
\newtheorem{remark}[theorem]{Remark}

\pgfdeclareshape{crosscircle}
{
  \inheritsavedanchors[from=circle] 
  \inheritanchorborder[from=circle]
  \inheritanchor[from=circle]{north}
  \inheritanchor[from=circle]{north west}
  \inheritanchor[from=circle]{north east}
  \inheritanchor[from=circle]{center}
  \inheritanchor[from=circle]{west}
  \inheritanchor[from=circle]{east}
  \inheritanchor[from=circle]{mid}
  \inheritanchor[from=circle]{mid west}
  \inheritanchor[from=circle]{mid east}
  \inheritanchor[from=circle]{base}
  \inheritanchor[from=circle]{base west}
  \inheritanchor[from=circle]{base east}
  \inheritanchor[from=circle]{south}
  \inheritanchor[from=circle]{south west}
  \inheritanchor[from=circle]{south east}
  \inheritbackgroundpath[from=circle]
  \foregroundpath{
    \centerpoint%
    \pgf@xc=\pgf@x%
    \pgf@yc=\pgf@y%
    \pgfutil@tempdima=\radius%
    \pgfmathsetlength{\pgf@xb}{\pgfkeysvalueof{/pgf/outer xsep}}%
    \pgfmathsetlength{\pgf@yb}{\pgfkeysvalueof{/pgf/outer ysep}}%
    \ifdim\pgf@xb<\pgf@yb%
      \advance\pgfutil@tempdima by-\pgf@yb%
    \else%
      \advance\pgfutil@tempdima by-\pgf@xb%
    \fi%
    \pgfpathmoveto{\pgfpointadd{\pgfqpoint{\pgf@xc}{\pgf@yc}}{\pgfqpoint{-0.707107\pgfutil@tempdima}{-0.707107\pgfutil@tempdima}}}
    \pgfpathlineto{\pgfpointadd{\pgfqpoint{\pgf@xc}{\pgf@yc}}{\pgfqpoint{0.707107\pgfutil@tempdima}{0.707107\pgfutil@tempdima}}}
    \pgfpathmoveto{\pgfpointadd{\pgfqpoint{\pgf@xc}{\pgf@yc}}{\pgfqpoint{-0.707107\pgfutil@tempdima}{0.707107\pgfutil@tempdima}}}
    \pgfpathlineto{\pgfpointadd{\pgfqpoint{\pgf@xc}{\pgf@yc}}{\pgfqpoint{0.707107\pgfutil@tempdima}{-0.707107\pgfutil@tempdima}}}
  }
}
\makeatother

\colorlet{symbols}{blue!90!black}
\colorlet{testcolor}{green!60!black}

\def\symbol#1{\textcolor{symbols}{#1}}
\def\1{\mathbf{\symbol{1}}}

\usetikzlibrary{shapes.misc}
\usetikzlibrary{shapes.symbols}
\usetikzlibrary{snakes}
\usetikzlibrary{decorations}
\usetikzlibrary{decorations.markings}

\def\drawx{\draw[-,solid] (-3pt,-3pt) -- (3pt,3pt);\draw[-,solid] (-3pt,3pt) -- (3pt,-3pt);}
\tikzset{
	root/.style={circle,fill=testcolor,inner sep=0pt, minimum size=2mm},
	dot/.style={circle,fill=black,inner sep=0pt, minimum size=1mm},
	var/.style={circle,fill=black!10,draw=black,inner sep=0pt, minimum size=2mm},
	dotred/.style={circle,fill=black!50,inner sep=0pt, minimum size=2mm},
	generic/.style={semithick,shorten >=1pt,shorten <=1pt},
	dist/.style={ultra thick,draw=testcolor,shorten >=1pt,shorten <=1pt},
	testfcn/.style={ultra thick,testcolor,shorten >=1pt,shorten <=1pt,<-},
	testfcnx/.style={ultra thick,testcolor,shorten >=1pt,shorten <=1pt,<-,
		postaction={decorate,decoration={markings,mark=at position 0.6 with {\drawx}}}},
	kprime/.style={semithick,shorten >=1pt,shorten <=1pt,densely dashed,->},
	kprimex/.style={semithick,shorten >=1pt,shorten <=1pt,densely dashed,->,
		postaction={decorate,decoration={markings,mark=at position 0.4 with {\drawx}}}},
	kernel/.style={semithick,shorten >=1pt,shorten <=1pt,->},
	multx/.style={shorten >=1pt,shorten <=1pt,
		postaction={decorate,decoration={markings,mark=at position 0.5 with {\drawx}}}},
	kernelx/.style={semithick,shorten >=1pt,shorten <=1pt,->,
		postaction={decorate,decoration={markings,mark=at position 0.4 with {\drawx}}}},
	kernel1/.style={->,semithick,shorten >=1pt,shorten <=1pt,postaction={decorate,decoration={markings,mark=at position 0.45 with {\draw[-] (0,-0.1) -- (0,0.1);}}}},
	kernel2/.style={->,semithick,shorten >=1pt,shorten <=1pt,postaction={decorate,decoration={markings,mark=at position 0.45 with {\draw[-] (0.05,-0.1) -- (0.05,0.1);\draw[-] (-0.05,-0.1) -- (-0.05,0.1);}}}},
	kernelBig/.style={semithick,shorten >=1pt,shorten <=1pt,decorate, decoration={zigzag,amplitude=1.5pt,segment length = 3pt,pre length=2pt,post length=2pt}},
	rho/.style={dotted,semithick,shorten >=1pt,shorten <=1pt},
	renorm/.style={shape=circle,fill=white,inner sep=1pt},
	labl/.style={shape=rectangle,fill=white,inner sep=1pt},
	xi/.style={circle,fill=symbols!10,draw=symbols,inner sep=0pt,minimum size=1.2mm},
	xix/.style={crosscircle,fill=symbols!10,draw=symbols,inner sep=0pt,minimum size=1.2mm},
	xib/.style={circle,fill=symbols!10,draw=symbols,inner sep=0pt,minimum size=1.6mm},
	xibx/.style={crosscircle,fill=symbols!10,draw=symbols,inner sep=0pt,minimum size=1.6mm},
	not/.style={circle,fill=symbols,draw=symbols,inner sep=0pt,minimum size=0.5mm},
	>=stealth,
	}
\makeatletter
\def\DeclareSymbol#1#2#3{\expandafter\gdef\csname MH@symb@#1\endcsname{\tikz[baseline=#2,scale=0.15,draw=symbols]{#3}}\expandafter\gdef\csname MH@symb@#1s\endcsname{\scalebox{0.7}{\tikz[baseline=#2,scale=0.15,draw=symbols]{#3}}}}
\def\<#1>{\csname MH@symb@#1\endcsname}
\makeatother

\DeclareSymbol{circle}{0.5}{\draw (0,0.7) node[xi] {};}
\DeclareSymbol{line}{0.5}{\draw (0,0.2) node[not] {} -- (0,1.4) node[not] {};}

\DeclareSymbol{Psi}{0.5}{\draw (0,0) node[not] {} -- (0,1.5) node[xi] {};}
\DeclareSymbol{Psi2}{0.5}{\draw (-0.8,1) node[xi] {} -- (0,0) node[not] {} -- (0.8,1) node[xi] {};}
\DeclareSymbol{IPsi2}{0}{\draw (0,1) -- (0.8,2.2) node[xi] {};\draw (0,-0.25) node[not] {} -- (0,1) node[not] {} -- (-0.8,2.2) node[xi] {};}
\DeclareSymbol{PsiIPsi2}{0}{\draw (0,1) -- (0.8,2.2) node[xi] {};\draw (0,-0.25) node[not] {} -- (0,1) node[not] {} -- (-0.8,2.2) node[xi] {};\draw (0,-0.25) node[not] {} -- (0.8,1) node[xi] {};}

\date{\today}

\title{On a non-linear 2D fractional wave equation}

\begin{document}

\

\begin{center}
{\large\textbf{
On a non-linear 2D fractional wave equation
}}\\~\\
Aur\'elien Deya\footnote{Institut \'Elie Cartan, Universit\' e de Lorraine, BP 70239, 54506 Vandoeuvre-l\`es-Nancy, France. \\Email: {\tt aurelien.deya@univ-lorraine.fr}}
\end{center}

\bigskip

{\small \noindent {\bf Abstract:} We pursue the investigations initiated in \cite{deya-wave} about a wave-equation model with quadratic perturbation and stochastic forcing given by a space-time fractional noise. We focus here on the two-dimensional situation and therein extend the results of \cite{deya-wave} to a rougher noise, through the use of a third-order expansion. We also point out the limits of the Wick-renormalisation procedure in this case.

\bigskip


\bigskip

}

\small

\section{Introduction and main results}

This paper can be seen as the continuation of the analysis carried out in \cite{deya-wave}. Just as in the latter reference, we consider the following non-linear stochastic wave model: 
\begin{equation}\label{dim-d-quadratic-wave}
\left\{
\begin{array}{l}
\partial^2_t u-\Delta u= u^2+\dot{B} \quad , \quad  t\in [0,T] \ , \ x\in \R^d \ ,\\
u(0,.)=\phi_0 \quad , \quad \partial_t u(0,.)=\phi_1 
\end{array}
\right.
\end{equation}
where $\phi_0,\phi_1$ are (deterministic) initial conditions in an appropriate Sobolev space and $\dot{B}\triangleq\partial_t\partial_{x_1}\dots  \partial_{x_{d}} B$, for some space-time fractional Brownian noise $B$. To be more specific, we will here focus on the two-dimension case (i.e., $d=2$ in (\ref{dim-d-quadratic-wave})) and in this situation, we wish to extend the considerations and results of \cite{deya-wave} to a rougher noise. For a clear expression of this roughness property, as well as a more explicit stochastic analysis, we have decided, just as in \cite{deya-wave}, to restrict our attention to a fractional Brownian noise (the most natural extension of the space-time white noise):    

\begin{definition}
Let $(\Omega,\mathcal{F},\mathbb{P})$ be a complete filtered probability space. For $H=(H_0,H_1,H_2) \in (0,1)^{3}$, we call a space-time fractional Brownian motion (or a fractional Brownian sheet) of Hurst index $H$ any centered Gaussian process $B:\Omega \times ([0,T]\times \R^2) \to \R$ whose covariance function is given by 
$$\mathbb{E} \big[ B_s(x_1,x_2)B_t(y_1,y_2)\big] = R_{H_0}(s,t) R_{H_1}(x_1,y_1)R_{H_2}(x_2,y_2) \ ,$$
where
$$R_{H_i}(x,y)\triangleq\frac12 (|x|^{2H}+|y|^{2H}-|x-y|^{2H}) \ .$$
\end{definition}

With this notation in hand, let us recall that the results of \cite{deya-wave} for the model (\ref{dim-d-quadratic-wave}) (with $d=2$) rely on some second-order strategy and only cover the case where
\begin{equation}\label{old-conditions}
 H_0+H_1+H_2> \frac54 \ .
\end{equation}
Our aim here is to go one step further and treat the situation where 
\begin{equation}\label{extension-h}
1< H_0+H_1+H_2\leq \frac54 \ ,
\end{equation}
a condition which, in some sense, will turn out to be optimal (see Proposition \ref{prop:limit-case} below). Beyond the extension result itself, the study will give us the opportunity to settle a third-order procedure, which somehow generalizes the so-called Da Prato-Debussche trick and involves the construction of sophisticated third-order stochastic objects. As far as we know, such a third-order expansion, inspired by the recent developments in the parabolic setting (see e.g. \cite[Section 1.1]{mourrat-weber}), is new in the stochastic wave literature (including the \enquote{white-noise} wave literature). 

\

In order to illustrate our strategy, and also to understand the whole difficulty raised by the transition from (\ref{old-conditions}) to (\ref{extension-h}), let us start with a few heuristic considerations on the quadratic fractional model (\ref{dim-d-quadratic-wave}).

\subsection{Heuristic considerations}\label{subsec:heuri}
For the sake of clarity, the following shortcut notations will be used throughout the paper: for all domain $D\subset \R^2$, all $T>0$, $\al\in \R$ and $p,q\geq 1$, we set  
\begin{equation}
L^q_T \cw^{\al,p} \triangleq L^p([0,T];\cw^{\al,p}(\R^2)) \quad , \quad L^q_T \cw^{\al,p}_D \triangleq L^p([0,T];\cw^{\al,p}(D)) \ ,
\end{equation}
where $\cw^{\al,p}(\R^2)$, resp. $\cw^{\al,p}(D)$, stands for the classical Sobolev space, resp. localized Sobolev space. 

\smallskip

Let us here fix a compact domain $D\subset \R^2$. At a formal level, the mild form of Equation (\ref{dim-d-quadratic-wave}) (with $d=2$) is given by
\begin{equation}\label{equa-u-heuri}
u=\partial_t (G_t\ast \phi_0)+G_t\ast \phi_1+G \ast u^2+\<Psi> \ ,
\end{equation}
where $G$ stands for the wave kernel in $\R^2$, characterized by its Fourier transform 
$$\fouri_x (G_t )(\xi)=\frac{\sin(t|\xi|)}{|\xi|} \ , \quad t\geq 0 \ , \ \xi\in \R^2 \ ,$$
and the symbol $\<Psi>$ refers to the associated \enquote{homogeneous} solution, that is the solution of
\begin{equation}\label{equation-psi}
\left\{
\begin{array}{l}
\partial^2_t \<Psi> -\Delta \<Psi> =\dot{B} \quad , \quad  t\in [0,T] \ , \ x\in \R^2 \ ,\\ 
\<Psi> (0,.)=0\quad , \quad \partial_t \<Psi> (0,.)=0 \ .
\end{array}
\right. 
\end{equation}
Of course, due to the roughness of the noise $\dot{B}$, the interpretation of the latter equation is not exactly a standard issue: in fact, for a proper definition of $\<Psi>$, we will rely in the sequel on a specific approximation procedure (see the convergence statement for the first component in Proposition \ref{prop:stoch-constr}). As a result of this procedure, $\<Psi>$ will (almost surely) appear as a stochastic process with values in the space $L^\infty_T\cw^{-\al,p}_D$, for all $p\geq 2$ and $\al >\frac32-(H_0+H_1+H_2)$. With condition (\ref{extension-h}) in mind, let us assume from now on that $\al>0$.

\smallskip

Going back to (\ref{equa-u-heuri}), a first natural idea toward a possible fixed-point argument (often referred to as the Da Prato-Debussche trick) is to consider the dynamics satisfied by the difference process $v \triangleq u-\<Psi>$. To this end, we can rewrite the equation as
\begin{equation}\label{equa-v-heuri-1}
v=\partial_t (G_t\ast \phi_0)+G_t\ast \phi_1+G \ast v^2+2\, G \ast (v\cdot  \<Psi>)+G \ast (\<Psi>)^2 \ ,
\end{equation}
and then try to identify some possible stable space for $v$. Observe however that we must here face with a new interpretation issue. Indeed, as $\<Psi>$ is expected to take values in some negative-order Sobolev space, it is not clear to know a priori how we must interpret the product $(\<Psi>)^2$ in (\ref{equa-v-heuri-1}). This problem has been treated in \cite{gubinelli-koch-oh} for the white-noise situation $H_0=H_1=H_2=\frac12$, and then in \cite{deya-wave} under the more general condition $H_0+H_1+H_2>\frac54$. In both of these references, it is shown that $(\<Psi>)^2$ can only be understood in some Wick sense, that is through a renormalization procedure and the use of stochastic estimates, and the construction gives birth to a process with values in $L^\infty_T\cw^{-2\al,p}_D$ (a.s.). A first new result of the present study will consist in extending this property under condition (\ref{extension-h}) (see the convergence result for the second component in Proposition \ref{prop:stoch-constr}). Let us henceforth denote by 
$$\<Psi2>\in L^\infty_T\cw^{-2\al,p}_D$$
the resulting Wick interpretation of the product $(\<Psi>)^2$, and consider the related equation
\begin{equation}\label{equa-v}
v=\partial_t (G_t\ast \phi_0)+G_t\ast \phi_1+G \ast v^2+2\, G \ast (v\cdot  \<Psi>)+G \ast \<Psi2> \ .
\end{equation}
At this point, recall that, even if the property is less convenient to handle than its parabolic counterpart, convolution with the wave kernel $G$ still yields some regularizing effect, as summed up through the so-called Strichartz inequalities (see e.g. \cite[Propositions 3.2 and 3.3]{deya-wave}). According to these results and using the symbol $\<IPsi2>$ for the convolution $G \ast \<Psi2>$, we can (at least morally) expect to have 
$$ \<IPsi2>\in L^\infty_T\cw^{1-2\al,2}_D\ .$$
Actually, just as with $\<Psi>$ and $\<Psi2>$, the use of stochastic arguments will allow us to improve this regularity property and derive (see the third component in Proposition \ref{prop:stoch-constr}) that
\begin{equation}\label{regularity-ipsi2}
\<IPsi2> \in L^\infty_T\cw^{1-2\al,p}_D \quad \text{for all} \ p\geq 2 \ .
\end{equation}
Going back to equation (\ref{equa-v}), we can then expect the solution $v$ to inherit the regularity of $\<IPsi2>$, so that, if we refer to standard distribution-theory results (such as those in the subsequent Proposition \ref{prop:product}), the product $v\cdot \<Psi>$ involved in the equation is likely to make sense as long as $(1-2\al)+(-\al)>0$, that is as long as $\al\in (0,\frac13)$. These heuristic arguments somehow account for the success of the \enquote{first-order} expansion used in \cite{deya-wave,gubinelli-koch-oh} (in both references, it is assumed that $\al\in (0,\frac14)$).

\smallskip

In order to go one step further and handle the situation where $\al \in (\frac13,\frac12)$ (which encompasses condition (\ref{extension-h})), let us iterate the above principles and consider the new process $w\triangleq v-\<IPsi2>$. Equation (\ref{equa-v}) now turns into
\begin{eqnarray}
w&=&\partial_t (G_t\ast \phi_0)+G_t\ast \phi_1+ G \ast \big( w+\<IPsi2>\big)^2+2\, G \ast \big( \big(w+\<IPsi2>\big) \<Psi>\big)\nonumber \\
&=&\partial_t (G_t\ast \phi_0)+G_t\ast \phi_1+G\ast w^2+G \ast \big(\<IPsi2>\big)^2+2\, G\ast \big(  w \cdot \<IPsi2> \big)+2\, G \ast \big( w \cdot \<Psi>\big) +2\, G \ast \<PsiIPsi2> \ ,\label{equa-w}
\end{eqnarray}
where the \enquote{third-order} process $\<PsiIPsi2> $ is defined as  
$$\<PsiIPsi2>  \triangleq \<IPsi2> \cdot \<Psi>\ .$$
Again, the interpretation of the latter product is not clear when $\al\in (\frac13,\frac12)$ (since $(1-2\al)+(-\al)<0$), but, just as with $\<Psi2>$, we can hope for a \enquote{stochastic construction} of this product, leading (a.s.) to a well-defined element
$$\<PsiIPsi2> \in  L^\infty_T\cw^{-\al,p}_D \ .$$
The consideration of this process will be one of the main novelty of our study. Its construction above the fractional noise corresponds to the convergence result for the fourth component in Proposition \ref{prop:stoch-constr}. Once endowed with $\<PsiIPsi2>$, and along the same pattern as above, we can (morally) expect to have
$$w \in L^\infty_T\cw^{1-\al,2}_D \ ,$$
so that the product $w\cdot  \<Psi>$ in (\ref{equa-w}) can make sense for any $\al\in (\frac13,\frac12)$ (due to $(1-\al)+(-\al)>0$), leading finally to the possibility of a well-posed equation in our setting. The details of this deterministic procedure will be the topic of Section \ref{subsec:solv-equa} below. Let us again emphasize the fact that, to the best of our knowledge, the use of such a third-order strategy cannot be found in the existing stochastic wave literature.

\smallskip

To end with these heuristic considerations, we need to specify that, just as in \cite{deya-wave}, we only aim at a local solution in this study, in time {\it as well as in space} (say on the fixed space domain $\{|x|\leq 1\}$). Let us thus introduce, for the rest of the paper, a smooth function $\rho: \R^2 \to \R$ with support included in $\{|x|\leq 2\}$ such that $\rho(x)=1$ for $|x|\leq 1$, and note that we will rather consider the following \enquote{localized} version of (\ref{equa-w}): 
\begin{align*}
&w=\partial_t (G_t\ast \phi_0)+G_t\ast \phi_1\\
&\hspace{0.5cm}+ G\ast \big( \rho^2w^2\big)+G \ast \big( \rho^2\big(\<IPsi2>\big)^2\big)+2\, G\ast \big( (\rho w) \cdot \big( \rho \<IPsi2>\big) \big)+2\, G \ast \big( (\rho w) \cdot \big( \rho \<Psi>\big) \big)+2\, G \ast \big( \rho^2\<PsiIPsi2>\big) \ .
\end{align*}

\begin{remark}
The use of the above symbols $\<Psi>$, $\<Psi2>$, $\<IPsi2>$ and $\<PsiIPsi2>$ is directly inspired by the recent literature on parabolic SPDEs (see e.g. \cite[Section 4]{hairer-pardoux} or \cite[Section 1]{mourrat-weber}). In fact, as the reader may have guessed it, the construction of these symbols (morally) obeys the following simple rules: $(i)$ the circle symbol $\<circle>$ refers to the noise $\dot{B}$; $(ii)$ each line $\<line>$ corresponds to a convolution with the wave kernel $G$; $(iii)$ two (sub-)symbols attached to a same black knot are just multiplied between each other (that is, the associated processes are multiplied between each other). Note however that the latter \enquote{code} is only heuristic: in particular, it does not take the possible renormalization procedures into account (as the one involved in the \enquote{Wick} definition of $\<Psi2>$).
\end{remark}

\subsection{Stochastic constructions}\label{subsec:stoch-results}

As evoked earlier, a proper definition of the stochastic components 
$$(\<Psi>,\<Psi2>,\<IPsi2>,\<PsiIPsi2>)$$
at the center of the above analysis will be obtained through an approximation procedure. Namely, starting from some smooth approximation $\dot{B}^n$ of the fractional noise, we first observe that, for each fixed $n\geq 1$, the homogeneous equation
\begin{equation}\label{equation-psi-n}
\left\{
\begin{array}{l}
\partial^2_t \<Psi>^n -\Delta \<Psi>^n =\dot{B}^n \quad , \quad  t\in [0,T] \ , \ x\in \R^2 \ ,\\ 
\<Psi>^n (0,.)=0\quad , \quad \partial_t \<Psi>^n (0,.)=0 \ 
\end{array}
\right.
\end{equation}
falls within the class of standard hyperbolic systems, for which a unique (global) solution $\<Psi>^n$ is known to exist. Then we define the (smooth) approximated path $\mathbf{\Psi}^n\triangleq \big(\<Psi>^n,\<Psi2>^n,\<IPsi2>^n,\<PsiIPsi2>^n\big)$ along the explicit formulas
\begin{equation}\label{approx-rp}
\<Psi2>^n(t,y)\triangleq (\<Psi>^n(t,y))^2-\mathbb{E}\big[ (\<Psi>^n(t,y))^2\big] \quad , \quad \<IPsi2>^n \triangleq G\ast \<Psi2>^n \quad , \quad \<PsiIPsi2>^n \triangleq \<IPsi2>^n \<Psi>^n \ ,
\end{equation}
and study the (almost sure) convergence of these processes in suitable spaces.

\smallskip

In order to implement this standard procedure, we will consider here the approximation derived from the so-called harmonizable representation of the space-time fractional Brownian motion, that is the formula (valid for every $H=(H_0,H_1,H_2)\in (0,1)^{3}$)  
\begin{equation*}
B_t(x_1,x_2)= c_{H} \int_{\xi\in \R}\int_{\eta \in \R^2} \widehat{W}(d\xi,d\eta) \frac{e^{\imath t\xi}-1}{|\xi|^{H_0+\frac12}} \frac{e^{\imath x_1 \eta_1 }-1}{|\eta_1|^{H_1+\frac12}}\frac{e^{\imath x_2 \eta_2 }-1}{|\eta_2|^{H_2+\frac12}} \ ,
\end{equation*}
where $c_H >0$ is a suitable constant and $\widehat{W}$ stands for the Fourier transform of a space-time white noise in $\R^{3}$, defined on some complete filtered probability space $(\Omega,\mathcal{F},\mathbb{P})$. The approximation $(B^n)_{n\geq 0}$ of $B$ is now defined as $B^0 \equiv 0$ and, for $n\geq 1$,
\begin{equation}
B^n_t(x_1,x_2)\triangleq c_{H} \int_{|\xi|\leq 2^n}\int_{|\eta|\leq 2^n} \widehat{W}(d\xi,d\eta) \frac{e^{\imath t\xi}-1}{|\xi|^{H_0+\frac12}} \frac{e^{\imath x_1 \eta_1 }-1}{|\eta_1|^{H_1+\frac12}}\frac{e^{\imath x_2 \eta_2 }-1}{|\eta_2|^{H_2+\frac12}} \ .
\end{equation}
It is readily checked that for all fixed $H=(H_0,H_1,H_2)\in (0,1)^{3}$ and $n\geq 1$, the process $B^n$ indeed corresponds  (almost surely) to a smooth function with respect to all its parameters, and we can thus consider the process $\mathbf{\Psi}^n$ associated with its derivative $\dot{B}^n$ (along (\ref{equation-psi-n})-(\ref{approx-rp})) .

\smallskip
 
With this notation in mind, our main stochastic result can be stated as follows:

\begin{proposition}\label{prop:stoch-constr}
Let $(H_0,H_1,H_2)\in (0,1)^{3}$ be such that 
\begin{equation}\label{constraint-h-i}
0<H_1<\frac34 \quad , \quad 0<H_2< \frac34 \quad , \quad 1 < H_0+H_1+H_2 \leq  \frac54 \ ,
\end{equation}
Then, for all compact domain $D\subset \R^2$, all $T>0$, $p\geq 2$ and 
\begin{equation}\label{regu-2}
\al >\frac32- (H_0+H_1+H_2)\ ,
\end{equation}
the sequence $(\mathbf{\Psi}^n)_{n\geq 1} \triangleq (\<Psi>^n,\<Psi2>^n,\<IPsi2>^n,\<PsiIPsi2>^n)_{n\geq 1}$ defined by (\ref{equation-psi-n})-(\ref{approx-rp}) converges to a limit in the space $L^p(\Omega;\mathcal{E}^{\al,p}_{T,D})$, where
\begin{equation}\label{defi:e-al-p}
\mathcal{E}^{\al,p}_{T,D} \triangleq L^\infty_T\mathcal{W}^{-\al,p}_D\times L^\infty_T\mathcal{W}^{-2\al,p}_D \times L^\infty_T\mathcal{W}^{1-2\al,p}_D\times L^\infty_T\mathcal{W}^{-\al,p}_D \ .
\end{equation}
We denote this limit by $\mathbf{\Psi}=(\<Psi>,\<Psi2>,\<IPsi2>,\<PsiIPsi2>)$.
\end{proposition}

\

The condition $H_0+H_1+H_2 >1$ that appears in (\ref{constraint-h-i}) turns out to be optimal in the framework of our strategy, as shown by the following divergence property:
\begin{proposition}\label{prop:limit-case}
Let $(H_0,H_1,H_2)\in (0,1)^{3}$ be such that 
\begin{equation}\label{constraint-h-i-opt}
H_0+H_1+H_2 \leq  1 \ .
\end{equation}
Then, for all $\al >0$ and $t>0$, one has $\mathbb{E}\Big[ \|\<Psi2>^n(t,.)\|_{\mathcal{W}^{-2\al,2}(D)}^2\Big] \to \infty$ as $n\to \infty$.
\end{proposition}
This result thus points out the limit of the Wick-renormalization procedure within our approach. In other words, any attempt to handle the rougher situation $H_0+H_1+H_2 \leq  1$ through the strategy displayed in Section \ref{subsec:heuri} should lean on a more sophisticated renormalization method. 

\

\subsection{Main results}

Let us recall that, for the rest of the paper, we have fixed a smooth function $\rho: \R^2 \to \R$ with support included in $\{|x|\leq 2\}$ and satisfying $\rho(x)=1$ for $|x|\leq 1$. Based on the considerations of Section \ref{subsec:heuri} and relying on the result of Proposition \ref{prop:stoch-constr}, the following definition of a (local-in-space) solution for (\ref{dim-d-quadratic-wave}) naturally arises:

\begin{definition}
Fix $(H_0,H_1,H_2)\in (0,1)^{3}$ such that condition (\ref{constraint-h-i}) is satisfied, and for every $T>0$, let 
$$\mathbf{\Psi}=(\<Psi>,\<Psi2>,\<IPsi2>,\<PsiIPsi2>)\in \bigcap_{\al>\frac32-(H_0+H_1+H_2)}\bigcap_{p\geq 2}  L^p(\Omega;\mathcal{E}^{\al,p}_{T,D})$$
be the process defined through Proposition \ref{prop:stoch-constr}, with $D\triangleq \{|x|\leq 2\}$. A stochastic process $(u(t,x))_{t\in [0,T],x\in \R^2}$ is said to be a Wick solution (on $[0,T]$) of the equation
\begin{equation}\label{2-d-quadratic-wave-defi}
\left\{
\begin{array}{l}
\partial^2_t u-\Delta u = \rho^2u^2 +\dot{B} \quad  , \quad x\in \R^2 \ ,\\
u(0,.)=\phi_0 \quad , \quad \partial_t u(0,.)=\phi_1 \ ,
\end{array}
\right.
\end{equation}
if, almost surely, the auxiliary process $w\triangleq u-\<Psi>-\<IPsi2>$ is a mild solution (on $[0,T]$) of the equation 
\begin{align}
&w=\partial_t (G_t\ast \phi_0)+G_t\ast \phi_1\nonumber\\
&\hspace{0.5cm}+ G\ast \big( \rho^2w^2\big)+G \ast \big( \rho^2\big(\<IPsi2>\big)^2\big)+2\, G\ast \big( (\rho w) \cdot \big( \rho \<IPsi2>\big) \big)+2\, G \ast \big( (\rho w) \cdot \big( \rho \<Psi>\big) \big)+2\, G \ast \big( \rho^2\<PsiIPsi2>\big) \ .\label{mild-equation-definition}
\end{align}
\end{definition}

\

Our main results regarding equation (\ref{2-d-quadratic-wave-defi}) finally read as follows:
\begin{theorem}\label{main-theo}
Let $(\phi_0,\phi_1) \in \cw^{\frac12,2}(\R^d) \times \cw^{-\frac12,2}(\R^d)$ and let $(H_0,H_1,H_2)\in (0,1)^{3}$ be such that 
$$
0<H_1<\frac34 \quad , \quad 0<H_2< \frac34 \quad , \quad 1 < H_0+H_1+H_2 \leq  \frac54 \ .
$$
Then, almost surely, there exists a time $T_0>0$ such that Equation (\ref{2-d-quadratic-wave-defi}) admits a unique Wick solution $u$ in the set 
\begin{equation}\label{defi-set-s-t}
\cs_{T_0}\triangleq \, \<Psi>+\<IPsi2>   + L^\infty_{T_0}\cw^{\frac12,2}  \ .
\end{equation}
\end{theorem}

\

Using the continuity properties of equation (\ref{mild-equation-definition}) with respect to $\mathbf{\Psi}$ (see estimate (\ref{boun-gamma-2}) below), we are also able to offer the following sequential approach to the problem: 

\begin{theorem}\label{main-theo-lim}
Under the assumptions of Theorem \ref{main-theo}, consider the sequence $(u^n)_{n\geq 0}$ of (classical) solutions to the renormalized equation
\begin{equation}\label{approx-equation}
\left\{
\begin{array}{l}
\partial^2_t u^n-\Delta u^n = \rho^2 (u^n)^2-\si^n+(1-\rho^2) (\<Psi>^n)^2+\dot{B}^n \quad , \quad   x\in \R^2 \ ,\\
u^n(0,.)=\phi_0 \quad , \quad \partial_t u^n(0,.)=\phi_1 \ ,
\end{array}
\right.
\end{equation}
where $\si^n_t(x)=\si^n_t \triangleq \mathbb{E}\big[ (\<Psi>^n_t(x))^2\big]$. Then
\begin{equation}\label{estim-cstt}
\si^n_t\stackrel{n\to \infty}{\sim} c^1_{H}\, t \, 2^{2n(\frac32-(H_0+H_1+H_2))}
\end{equation}
for some constant $c^1_{H}$, and, almost surely, there exists a time $T_0>0$ and a subsequence of $(u^n)$ that converges in the space $L^\infty_{T_0}\cw^{-\al,2}_D$ to the solution $u$ exhibited in Theorem \ref{main-theo}, for every $\al >\frac32-(H_0+H_1+H_2)$.
\end{theorem}

\

\begin{remark}
0n the unit ball $\{|x|\leq 1\}$, that is on the space-domain we actually focus, equation (\ref{approx-equation}) reduces to the Wick-renormalized model $\partial^2_t u^n-\Delta u^n = [(u^n)^2-\si^n]+\dot{B}^n$, as expected. In fact, a global-in-space solution (i.e. the consideration of $\rho \equiv 1$ on $\R^2$ in the above results) could perhaps be obtained through the involvement of weighted spaces in the subsequent analysis, both in the convergence results of Proposition \ref{prop:stoch-constr} and in the study of the deterministic equation.  Nevertheless, we expect this extension to be the source of technical stability issues (see for instance the developments in \cite{hairer-labbe} for the parabolic setting) and therefore we postpone these investigations to a (possible) future publication. On the other hand, studying the equation on the 2D-torus (with appropriate boundary conditions) should certainly give rise to very similar estimates and results. This would however require the introduction of a fractional Brownian noise on the torus, a model that we find slightly more \enquote{exotic} than ours. 
\end{remark}

\begin{remark}
At the level of the approximated equation (\ref{approx-equation}), another (perhaps more natural) possibility would be to consider, just as in \cite[Theorem 1.7]{deya-wave}, the approximation given by 
\begin{equation}\label{approx-equation-alter}
\partial^2_t u^n-\Delta u^n = \rho^2 ((u^n)^2-\si^n)+\dot{B}^n \ .
\end{equation}
However, expanding equation (\ref{approx-equation-alter}) along the pattern of Section \ref{subsec:heuri} readily leads us to the consideration of the $\rho$-dependent path
$$\widetilde{\mathbf{\Psi}}^n =\big(\<Psi>^n,\<Psi2>^n,G\ast \big(\rho^2 \<Psi2>^n\big),\big[G\ast \big(\rho^2 \<Psi2>^n\big)\big] \<Psi>^n\big)$$
in place of the ($\rho$-independent) path $\mathbf{\Psi}^n =\big(\<Psi>^n,\<Psi2>^n,\<IPsi2>^n,\<PsiIPsi2>^n\big)$ at the core of the above analysis. We have thus preferred a more \enquote{intrinsic} expression for this central object, which explains our focus on (\ref{approx-equation}) (see Section \ref{subsec:proof-main} for more details on the transition from (\ref{approx-equation}) to (\ref{mild-equation-definition})). 
\end{remark}

\

The rest of the paper is organized along a natural two-part splitting. In Section \ref{sec:determ-eq}, we will handle the deterministic aspects of the problem, that is the well-posedness of equation (\ref{mild-equation-definition}) once endowed with a path $\mathbf{\Psi}=(\<Psi>,\<Psi2>,\<IPsi2>,\<PsiIPsi2>) \in  \mathcal{E}^{\al,p}_{T,D}$. The proofs of Theorem \ref{main-theo} and Theorem \ref{main-theo-lim} will almost immediately follow (Section \ref{subsec:proof-main}). Section \ref{sec:stoch} will then be devoted to the technical stochastic estimates behind the convergence statement of Proposition \ref{prop:stoch-constr} and the explosion phenomenon of Proposition \ref{prop:limit-case}.

\

Throughout the paper, we will denote by $A \lesssim B$ any estimate of the form $A\leq cB$, where $c >0$ is a constant that does not depend on the parameters under consideration.

\

\textbf{Acknowledgements}: I am grateful to Laurent Thomann for his encouragements and for his advice about the treatment of the deterministic equation.

\

\section{Auxiliary (deterministic) equation}\label{sec:determ-eq}
Let us first recall that for all parameters $\al >0$, $p\geq 1$, $T>0$, and for every domain $D\subset \R^2$, the functional space $\mathcal{E}^{\al,p}_{T,D}$ has been introduced in (\ref{defi:e-al-p}). Also, for the whole study, we have fixed a smooth cut-off functions $\rho$ with support in $\{|x|\leq 2\}$ (such that $\rho(x)=1$ for $|x|\leq 1$), and for this reason {\it let us set $D\triangleq \{|x|\leq 2\}$ for the rest of the section.}

Now, for fixed elements $\mathbf{\Psi}=(\<Psi>,\<Psi2>,\<IPsi2>,\<PsiIPsi2>)\in \mathcal{E}^{\al,p}_{T,D}$ (with suitable $\al,p$) and $\phi_0,\phi_1$ in appropriate Sobolev spaces, our objective here is to settle a (local) fixed-point argument for the mild equation
\begin{align}
&w=\partial_t (G_t\ast \phi_0)+G_t\ast \phi_1\nonumber\\
&\hspace{0.5cm}+ G\ast \big( \rho^2w^2\big)+G \ast \big( \rho^2\big(\<IPsi2>\big)^2\big)+2\, G\ast \big( (\rho w) \cdot \big( \rho \<IPsi2>\big) \big)+2\, G \ast \big( (\rho w) \cdot \big( \rho \<Psi>\big) \big)+2\, G \ast \big( \rho^2\<PsiIPsi2>\big) \ .\label{equa-w-local-deter}
\end{align}

\subsection{Basic preliminary results}

As we mentionned it in the introduction, convolving with the wave kernel is known to entail specific regularization effects, that are generally summed up through the so-called Strichartz inequalities (see \cite{ginibre-velo}). It turns out that the following (much) weaker result will be sufficient for our purpose here:

\begin{proposition}\label{prop:regu-wave}
$(i)$ For all $0\leq T \leq 1$ and $w\in L^1([0,T];\cw^{-\frac12,2}(\R^2))$, it holds that
\begin{equation}\label{basic-strichartz-1}
\cn\big[ G \ast w;L^\infty([0,T];\cw^{\frac12,2}(\R^2))\big] \lesssim \cn[w;L^1([0,T];\cw^{-\frac12,2}(\R^2))] \ .
\end{equation}
$(ii)$ For all $0\leq T \leq 1$ and $(\phi_0,\phi_1) \in \cw^{\frac12,2}(\R^2) \times \cw^{-\frac12,2}(\R^2)$, it holds that
\begin{equation}\label{basic-strichartz-2}
\cn[ \partial_t(G_t \ast \phi_0)+G_t\ast \phi_1 ;L^\infty([0,T];\cw^{\frac12,2}(\R^2))]\lesssim \| \phi_0\|_{\cw^{\frac12,2}(\R^2)}+\| \phi_1\|_{\cw^{-\frac12,2}(\R^2)} \ .
\end{equation}
\end{proposition}

\begin{proof}
Both (\ref{basic-strichartz-1}) and (\ref{basic-strichartz-2}) rely on elementary estimates only.

\smallskip

\noindent
$(i)$ Observe first that
$$\fouri_x \big( \big[ G\ast w\big](t,.)\big)(\xi)=\int_0^t ds \, \big( \fouri_x G_{t-s}\big)(\xi) \big( \fouri_x w_s\big) (\xi) \ ,$$
so that, for every $t\in [0,T]$,
\begin{eqnarray*}
\int_{\R^2} d\xi \, \{1+|\xi|^2\}^{\frac12} \big|\fouri_x\big( \big[ G\ast w\big](t,.)\big)(\xi)\big|^2 & \leq & \int_0^t ds \int_{\R^2} d\xi \, \{1+|\xi|^2\}^{\frac12} \frac{\sin^2((t-s)|\xi|)}{|\xi|^2} | ( \fouri_x w_s) (\xi)|^2\\
&\lesssim& \int_0^t ds \int_{\R^2} d\xi \, \{1+|\xi|^2\}^{-\frac12}  | (\fouri_x w_s)(\xi)|^2  \ ,
\end{eqnarray*}
hence the conclusion.

\smallskip

\noindent
$(ii)$ The bound for $\cn[ G_t\ast \phi_1 ;L^\infty([0,T];\cw^{\frac12,2}(\R^2))]$ follows from similar arguments as above. Then, along the same idea, observe that
$$\fouri_x\big( \partial_t(G_. \ast \phi_0)(t,.)\big)(\xi)=\partial_t ( \fouri_x G)_t(\xi) (\fouri_x \phi_0)(\xi)=\cos(t|\xi|) (\fouri_x \phi_0)(\xi) \ ,$$
and so, for every $t\in [0,T]$,
$$
\int_{\R^2} d\xi \, \{1+|\xi|^2\}^{-\frac12} \big|\fouri_x\big( \partial_t(G_. \ast \phi_0)(t,.)\big)(\xi)\big|^2  \leq  \int_{\R^2} d\xi \, \{1+|\xi|^2\}^{-\frac12} |(\fouri_x \phi_0)(\xi)|^2 \ .
$$

\end{proof}

Then, in order to control the product operations involved in (\ref{equa-w-local-deter}), we will appeal to the following standard properties, the proof of which can for instance be found in \cite[Chapter 4]{runst-sickel}:

\begin{proposition}\label{prop:product}
$(i)$ For all $\al,\be >0$ and $0<p,p_1,p_2\leq \infty$ such that
$$\frac{1}{p}\leq \frac{1}{p_1}+\frac{1}{p_2} \quad , \quad 0<\al<\be<\frac{2}{p_2} \quad , \quad \min\Big(\frac{2}{p}+\al,2\Big)>\Big( \frac{2}{p_1}+\al \Big)+\Big( \frac{2}{p_2}-\be \Big) \ ,$$
one has
$$
\| f\cdot g\|_{\cw^{-\al,p}(\R^2)} \lesssim \|f\|_{\cw^{-\al,p_1}(\R^2)} \| g\|_{\cw^{\be,p_2}(\R^2)} \ .
$$
$(ii)$ For all $\be >\frac{1}{2}$, $0<\al<\be$ and $p>2$, one has
$$
\| f\cdot g\|_{\cw^{-\al,p}(\R^2)} \lesssim \|f\|_{\cw^{-\al,p}(\R^2)} \| g\|_{\cw^{\be,2}(\R^2)} \ .
$$
$(iii)$ For all $\al>0$ and $0<p,p_1,p_2\leq \infty$ such that
$$0<\al<\min\bigg(\frac{2}{p_1},\frac{2}{p_2}\bigg) \quad \text{and} \quad  \frac{1}{p_1}+\frac{1}{p_2}=\frac{1}{p}<\al+1 \ ,$$ 
one has
$$
\| f\cdot g\|_{\cw^{\al,p}(\R^2)} \lesssim \|f\|_{\cw^{\al,p_1}(\R^2)} \| g\|_{\cw^{\al,p_2}(\R^2)} \ .
$$
\end{proposition}

\subsection{Solving the deterministic equation}\label{subsec:solv-equa}

We are now in a position to exhibit suitable bounds for (\ref{equa-w-local-deter}).
 
\begin{proposition}\label{prop:deterministic-result}
Fix $\al \in (\frac14,\frac12)$ and for all $T>0$, $\phi=(\phi_0,\phi_1) \in \cw^{\frac12,2}(\R^2) \times \cw^{-\frac12,2}(\R^2)$ and $\mathbf{\Psi}=\big(\<Psi>,\<Psi2>,\<IPsi2>,\<PsiIPsi2>\big)\in \ce^{\al,4}_{T,D}$, consider the map $\gga_{T,\phi,\mathbf{\Psi}}:L^\infty_T\cw^{\frac12,2}\to L^\infty_T\cw^{\frac12,2}$ given by 
\begin{align*}
&\gga_{T,\phi,\mathbf{\Psi}}(w)\triangleq \partial_t (G_t\ast \phi_0)+G_t\ast \phi_1\\
&\hspace{1.5cm}+G\ast \big( \rho^2w^2\big)+G \ast \big( \rho^2\big(\<IPsi2>\big)^2\big)+2\, G\ast \big( (\rho w) \cdot \big( \rho \<IPsi2>\big) \big)+2\, G \ast \big( (\rho w) \cdot \big( \rho \<Psi>\big) \big)+2\, G \ast \big( \rho^2\<PsiIPsi2>\big) \ .
\end{align*}
Then, for all $T>0$, $\phi=(\phi_0,\phi_1) \in \cw^{\frac12,2}(\R^2) \times \cw^{-\frac12,2}(\R^2)$, $\mathbf{\Psi}_1,\mathbf{\Psi}_2\in \mathcal{E}^{\al,4}_{T,D}$ and $w,w_1,w_2 \in L^\infty_T\cw^{\frac12,2}$, the following bounds hold true:
\begin{align}
&\cn\big[\Gamma_{T,\phi,\mathbf{\Psi}_1}(w);L^\infty_T\cw^{\frac12,2}\big] \lesssim \| \phi_0\|_{\cw^{\frac12,2}}+\| \phi_1\|_{\cw^{-\frac12,2}}\nonumber\\
&\hspace{3cm}+T\Big\{ \cn\big[w;L^\infty_T\cw^{\frac12,2}\big]^2+\|\mathbf{\Psi}_1\|^2+   \|\mathbf{\Psi}_1\|\, \cn\big[w;L^\infty_T\cw^{\frac12,2}\big]+\|\mathbf{\Psi}_1\|\Big\} \ ,\label{boun-gamma-1}
\end{align}
and
\begin{align}
&\cn\big[\Gamma_{T,\phi,\mathbf{\Psi}_1}(w_1)-\Gamma_{T,\phi,\mathbf{\Psi}_2}(w_2);L^\infty_T\cw^{\frac12,2}\big] \nonumber\\
&\lesssim T\Big[ \cn\big[w_1-w_2;L^\infty_T\cw^{\frac12,2}\big]\big\{\cn\big[w_1;L^\infty_T\cw^{\frac12,2}\big]+\cn\big[w_2;L^\infty_T\cw^{\frac12,2}\big]\big\} +   \|\mathbf{\Psi}_1-\mathbf{\Psi}_2\|\, \cn\big[w_1;L^\infty_T\cw^{\frac12,2}\big]\nonumber\\
& \hspace{6cm}+  \|\mathbf{\Psi}_2\|\, \cn\big[w_1-w_2;L^\infty_T\cw^{\frac12,2}\big]+T \|\mathbf{\Psi}_1-\mathbf{\Psi}_2\|\Big] \ , \label{boun-gamma-2}
\end{align}
where the proportional constants only depend on $s$ and the norm $\|.\|$ is naturally defined as
$$
\|\mathbf{\Psi}\|=\|\mathbf{\Psi}\|_{\mathcal{E}^{\al,4}_{T,D}}\triangleq \cn\big[\<Psi>;L^{\infty}_T\cw^{-\al,4}_D\big]+\cn\big[\<Psi2>;L^{\infty}_T\cw^{-2\al,4}_D\big]+\cn\big[\<IPsi2>;L^{\infty}_T\cw^{1-2\al,4}_D\big]+\cn\big[\<PsiIPsi2>;L^{\infty}_T\cw^{-\al,4}_D\big] \ .
$$
\end{proposition}

\

\begin{proof}
As expected, the two bounds (\ref{boun-gamma-1}) and (\ref{boun-gamma-2}) will follow from the combination of the estimates in Proposition \ref{prop:regu-wave} and Proposition \ref{prop:product}. The elementary Sobolev embedding
\begin{equation}\label{sobol-emb}
\cw^{s_0,p_0}(D) \subset \cw^{s_1,p_1}(D) \quad \text{if} \ s_0-\frac{2}{p_0} \geq s_1-\frac{2}{p_1} 
\end{equation}
will also be requested.

\

\noindent
\textbf{Initial conditions:} the bound for $\cn[ \partial_t(G_t \ast \phi_0)+G_t\ast \phi_1 ;L^\infty_T\cw^{\frac12,2}]$ immediately follows from Proposition \ref{prop:regu-wave}, item $(ii)$.

\

\noindent
\textbf{Bound on $G\ast (\rho^2 w^2)$:} Using successively (\ref{basic-strichartz-1}) and (\ref{sobol-emb}), we deduce that
$$
\cn\big[ G\ast (\rho^2 w^2);L^\infty_T\cw^{\frac12,2}\big]  \lesssim   \cn\big[\rho^2 w^2;L^1_T\cw^{-\frac12,2}_D\big] \lesssim \cn\big[\rho^2 w^2;L^1_TL^{\frac43}_D\big] \lesssim \cn\big[w;L^2_TL^{\frac83}_D\big]^2\lesssim T\cn\big[w;L^\infty_TL^{\frac83}_D\big]^2 \ ,
$$
and we get the expected bound through the embedding $\cw_D^{\frac12,2}\subset L_D^{\frac83}$.

\

\noindent
\textbf{Bound on $G\ast \big(\rho^2 \big( \<IPsi2>\big)^2\big)$:} Just as above, we have
$$\cn\big[ G\ast \big(\rho^2 \big( \<IPsi2>\big)^2\big);L^\infty_T\cw^{\frac12,2}\big] \lesssim T\cn\big[\<IPsi2>;L^\infty_TL^{\frac83}_D\big]^2 \ ,$$
and the desired bound is here obtained through the embedding $\cw^{1-2\al,4}_D\subset L_D^{\frac83}$. 

\smallskip

\noindent
\textbf{Bound on $G\ast \big( (\rho w) \cdot \big( \rho \<IPsi2>\big) \big)$:} 
Let $1<\rti_1<2$ be defined by the relation
$$\frac{1}{\rti_1}=\frac54-\al\ .$$
Using successively (\ref{basic-strichartz-1}) and (\ref{sobol-emb}), we get that
$$ \cn\big[ G\ast \big( (\rho w) \cdot \big( \rho \<IPsi2>\big) \big);L^\infty_T\cw^{\frac12,2}\big] \lesssim \cn\big[(\rho w) \cdot\big( \rho \<IPsi2>\big);L^{1}_T\cw^{-\frac12,2}_D\big]\lesssim T\cn\big[(\rho w) \cdot\big( \rho \<IPsi2>\big);L^{\infty}_T\cw^{1-2\al,\rti_1}_D\big] \ .$$
Then introduce the additional parameter $2\leq p_1\leq 4$ defined by
$$\frac{1}{p_1}=\frac{1}{\rti_1}-\frac12 =\frac34-\al\ .$$
By Proposition \ref{prop:product} (item $(iii)$), we know that for each fixed $t\geq 0$, 
\begin{eqnarray*}
\big\|(\rho w)_t \cdot\big( \rho \<IPsi2>\big)_t \big\|_{\cw^{1-2\al,\rti_1}(\R^2)}&\lesssim & \big\|(\rho w)_t  \big\|_{\cw^{1-2\al,2}(\R^2)} \big\|\big( \rho \<IPsi2>\big)_t \big\|_{\cw^{1-2\al,p_1}(\R^2)}\\
&\lesssim& \big\|w_t  \big\|_{\cw^{\frac12,2}_D} \big\|\<IPsi2>_t \big\|_{\cw^{1-2\al,4}_D} \ ,
\end{eqnarray*}
which immediately yields
$$
\cn\big[ G\ast \big( (\rho w) \cdot \big( \rho \<IPsi2>\big) \big);L^\infty_T\cw^{\frac12,2}\big]\lesssim T \cn\big[ w;L^{\infty}_T\cw^{\frac12,2}\big] \cn\big[\<IPsi2>;L^{\infty}_T\cw^{1-2\al,4}_D\big]\lesssim T \cn\big[w;L^{\infty}_T\cw^{\frac12,2}\big] \|\mathbf{\Psi}_1\| \ .
$$

\

\noindent
\textbf{Bound on $G\ast \big( (\rho w) \cdot \big( \rho \<Psi>\big) \big)$:}
Let $1<\rti_2<2$ be defined by the relation
$$\frac{1}{\rti_2}=\frac34-\frac{\al}{2} \ .$$
Using successively (\ref{basic-strichartz-1}) and (\ref{sobol-emb}), we get that
$$ \cn\big[ G\ast \big( (\rho w) \cdot \big( \rho \<Psi>\big) \big);L^\infty_T\cw^{\frac12,2}\big] \lesssim \cn\big[(\rho w) \cdot\big( \rho \<Psi>\big);L^{1}_T\cw^{-\frac12,2}_D\big]\lesssim T\cn\big[(\rho w) \cdot\big( \rho \<Psi>\big);L^{\infty}_T\cw^{-\al,\rti_2}_D\big] \ .$$
Then, using Proposition \ref{prop:product} (item $(i)$ and $(ii)$), we can assert that for each fixed $t\geq 0$, 
$$\big\|(\rho w)_t \cdot\big( \rho \<Psi>\big)_t \big\|_{\cw^{-\al,\rti_2}_D}\lesssim \big\|(\rho w)_t  \big\|_{\cw^{\frac12,2}_D} \big\|\big( \rho \<Psi>\big)_t \big\|_{\cw^{-\al,4}_D}\lesssim \big\| w_t  \big\|_{\cw^{\frac12,2}_D} \big\|\<Psi>_t \big\|_{\cw^{-\al,4}_D}$$
and accordingly
$$
\cn\big[ G\ast \big( (\rho w) \cdot \big( \rho \<Psi>\big) \big);L^\infty_T\cw^{\frac12,2}\big] \lesssim T \cn\big[ w;L^{\infty}_T\cw^{\frac12,2}\big] \|\mathbf{\Psi}_1\| \ .
$$

\

\noindent
\textbf{Bound on $G \ast \big( \rho^2\<PsiIPsi2>\big)$:} 
By (\ref{basic-strichartz-1}), 
$$ \cn\big[ G \ast \big( \rho^2\<PsiIPsi2>\big);L^\infty_T\cw^{\frac12,2}\big] \lesssim \cn\big[\rho^2\<PsiIPsi2>; L^1_T\cw_D^{-\frac12,2}\big]\lesssim \cn\big[\rho^2\<PsiIPsi2> ;L^\infty_T\cw_D^{-\al,4}\big]\lesssim T \|\mathbf{\Psi}_1\| \ ,$$
which concludes the proof of (\ref{boun-gamma-1}).

\

We can then show (\ref{boun-gamma-2}) along similar arguments. 
\end{proof}

\subsection{Proof of the main results}\label{subsec:proof-main}

\begin{proof}[Proof of Theorem \ref{main-theo}]
The combination of (\ref{boun-gamma-1}) and (\ref{boun-gamma-2}) easily allows us to assert that for $T_0>0$ small enough and for all $\mathbf{\Psi}\in \mathcal{E}^{\al,4}_{T_0,D}$ (with $\al\in (\frac14,\frac12)$), the map $\gga_{T_0,\phi,\mathbf{\Psi}}$ is a contraction on some stable subset of $L^\infty_{T_0}\cw^{\frac12,2}_D$, which yields a unique solution $w$ for equation (\ref{equa-w-local-deter}). Then it suffices of course to apply this result (in an almost sure way) to the stochastic process $\mathbf{\Psi}=\big( \<Psi>,\<Psi2>,\<IPsi2>,\<PsiIPsi2>\big)\in \mathcal{E}^{\al,4}_{T_0,D}$ given by Proposition \ref{prop:stoch-constr}, where $\al$ is picked such that $\frac14\leq \frac32-(H_0+H_1+H_2)<\al <\frac12$, and set $u\triangleq \<Psi>+\<IPsi2>+w$.
\end{proof}

\begin{proof}[Proof of Theorem \ref{main-theo-lim}]
For fixed $n\geq 1$, let $u^n$ be the solution of (\ref{approx-equation}) and set $w^n \triangleq u^n-\<Psi>^n-\<IPsi2>^n$, where $\<Psi>^n$ and $\<IPsi2>^n$ are defined in (\ref{approx-rp}). Following the lines of Section \ref{subsec:heuri}, we can explicitly verify that $w^n$ is the solution of (\ref{equa-w-local-deter}) associated with the process $(\mathbf{\Psi}^n)_{n\geq 1} \triangleq (\<Psi>^n,\<Psi2>^n,\<IPsi2>^n,\<PsiIPsi2>^n)$. Observe indeed that the following simplification occurs: setting $v^n \triangleq u^n-\<Psi>^n$, we have
\begin{eqnarray*}
v^n&=&\partial_t (G_t\ast \phi_0)+G_t\ast \phi_1+G \ast\big[ \rho^2 \big(v^n+\<Psi>^n\big)^2-\si^n+(1-\rho^2) \big( \<Psi>^n\big)^2 \big] \\
&=&\partial_t (G_t\ast \phi_0)+G_t\ast \phi_1+G \ast\big[ \rho^2 (v^n)^2+2 \rho^2 v^n \<Psi>^n \big]+G\ast \<Psi2>^n \ ,
\end{eqnarray*}
and from there it is readily checked that $w^n (=v^n-G\ast \<Psi2>^n)$ satisfies the expected equation (\ref{equa-w-local-deter}). Then, based on (\ref{boun-gamma-1})-(\ref{boun-gamma-2}), the convergence of (a subsequence of) $w^n$ in $L^\infty_{T_0}\cw^{\frac12,2}$ (for some $T_0>0$) can be shown with the very same arguments as those of the proof of \cite[Theorem 1.7]{deya-wave}, and the convergence of $u^n$ in $L^\infty_{T_0}\cw^{-\al,2}_D$ immediately follows. 

\smallskip

For the asymptotic estimate of $\si^n$, let us slightly anticipate the notations of Section \ref{sec:stoch}: in particular, using the forthcoming formula (\ref{cova-gene-bis}), we get that  
\begin{eqnarray*}
\si^n(t) \ = \ \mathbb{E}\big[ \big(\<Psi>^n_t(x)\big)^2\big]&=&c \int_{|\xi|\leq 2^n}  \frac{d\xi}{|\xi|^{2H_0-1}}\int_{|\eta|\leq 2^n}  \frac{d\eta_1 d\eta_2}{|\eta_1|^{2H_1-1}|\eta_2|^{2H_2-1}} \, |\ga_t(\xi,|\eta|)|^2 \\
&=&c \int_0^{2^n} \frac{dr}{r^{2(H_1+H_2)-3}}\int_{|\xi|\leq 2^n}  \frac{d\xi}{|\xi|^{2H_0-1}} \, |\ga_{t}(\xi,r)|^2 \ .
\end{eqnarray*}
Assertion (\ref{estim-cstt}) is now a straightforward consequence of \cite[Proposition 2.4]{deya-wave}.
\end{proof}

\section{Stochastic constructions}\label{sec:stoch}

Let us now turn to the main technical part of our analysis, namely the proofs of Propositions \ref{prop:stoch-constr} and \ref{prop:limit-case}, which includes in particular the construction of the central path $\mathbf{\Psi}=\big( \<Psi>,\<Psi2>,\<IPsi2>,\<PsiIPsi2>\big)$ above the fractional noise. To this end, our arguments will occasionally appeal to some of the technical results of \cite{deya-wave}. However, recall that, in comparison with the setting of \cite{deya-wave}, we are dealing with a rougher situation here and third-order processes must also come into the picture, so that new (sophisticated) estimates shall be required. 

\smallskip

Let us start with the introduction of a few convenient notations (related to the wave kernel and the fractional noise) that we will extensively use in the sequel. First, we set, for all $H=(H_0,H_1,H_2)\in (0,1)^3$, $\xi\in \R$, $\eta \in \R^2$ and $\rho,t\geq 0$,
\begin{equation}
\ga_t(\xi,\rho)\triangleq e^{\imath \xi t}\int_0^t ds \, e^{-\imath \xi s} \frac{\sin(s\rho)}{\rho} \quad , \quad \ga_{s,t}(\xi,\rho)\triangleq \ga_t(\xi,\rho)-\ga_s(\xi,\rho) \ ,
\end{equation}
\begin{equation}\label{n-h}
K^H(\eta)\triangleq \frac{|\eta_1|^{1-2H_1} |\eta_2|^{1-2H_2}}{1+|\eta|^{1+2H_0}} \ .
\end{equation}
For all $\tau\in \big\{\<Psi>,\<Psi2>,\<IPsi2>,\<PsiIPsi2> \big\}$, $1\leq n \leq m$ and  $0\leq s,t\leq 1$, let us also set $\tau^{n,m}_t \triangleq \tau^m_t-\tau^n_t$, and then, for $f\in \{\tau^{n},\tau^{m},\tau^{n,m}\}$, $f_{s,t} \triangleq f_t-f_s$.

\smallskip

With this notation in mind, the following \enquote{covariance} identity clearly holds true: for all $a=(a_1,a_2)$, resp. $b=(b_1,b_2)$, with $a_i\in \{n,m,\{n,m\}\}$, resp. $b_i\in \{s,t,\{s,t\}\}$, and all $y,\yti\in \R^2$, 
\begin{equation}\label{cova-gene}
\mathbb{E}\big[  \<Psi>^{a_1}_{b_1}(y)\overline{\<Psi>^{a_2}_{b_2}(\tilde{y})}\big]=c_H\int_{\R^2} d\eta \, e^{\imath \langle \eta,y\rangle}e^{-\imath \langle \eta,\yti\rangle} L^{H,a}_b(\eta) \ ,
\end{equation}
where 
$$L^{H,a}_b(\eta) \triangleq \frac{1}{|\eta_1|^{2H_1-1} |\eta_2|^{2H_2-1}} \int_{(\xi,\eta)\in \cd^{a_1} \cap \cd^{a_2}} d\xi \, \frac{\ga_{b_1}(\xi,|\eta|) \overline{\ga_{b_2}(\xi,|\eta|)}}{|\xi|^{2H_0-1}}  $$
with
$$\cd^n\triangleq \{|\xi|\leq 2^n \, , \, |\eta|\leq 2^n\} \quad , \quad  \cd^m\triangleq \{|\xi|\leq 2^m \, , \, |\eta|\leq 2^m\} \quad \text{and} \quad \cd^{n,m}\triangleq \cd^m \backslash \cd^n  \ .$$
In the same way, it holds that
\begin{equation}\label{cova-gene-bis}
\mathbb{E}\big[  \<Psi>^{a_1}_{b_1}(y)\<Psi>^{a_2}_{b_2}(\tilde{y})\big]=c_H\int_{\R^2} d\eta \, e^{\imath \langle \eta,y\rangle}e^{\imath \langle \eta,\yti\rangle} \tilde{L}^{H,a}_b(\eta)
\end{equation}
with
$$\tilde{L}^{H,a}_b(\eta) \triangleq \frac{1}{|\eta_1|^{2H_1-1} |\eta_2|^{2H_2-1}} \int_{(\xi,\eta)\in \cd^{a_1} \cap \cd^{a_2}} d\xi \, \frac{\ga_{b_1}(\xi,|\eta|) \ga_{b_2}(\xi,|\eta|)}{|\xi|^{2H_0-1}} \ .$$

\

Our estimates toward Proposition \ref{prop:stoch-constr} and Proposition \ref{prop:limit-case} will heavily rely on the following bounds for $L^{H,a}_b$:
\begin{lemma}\label{lem:bou-l-h}
For all $H=(H_0,H_1,H_2)\in (0,1)^3$, $\varepsilon\in (0,\min(H_0,\frac14))$, $0\leq n\leq m$, $0\leq s,t,u \leq 1$ and $\eta\in \R^2$, it holds that
\begin{equation}\label{bou-l-1}
\max \big( \big| L^{H,(m,m)}_{t,t}(\eta) \big|,\big| \tilde{L}^{H,(m,m)}_{t,t}(\eta) \big|\big) \lesssim K^{H_{0,\varepsilon}}(\eta) 
\end{equation}
and 
\begin{equation}\label{bou-l-2}
\max\big( \big| L^{H,((n,m),m)}_{(s,t),t}(\eta) \big|,\big| \tilde{L}^{H,((n,m),m)}_{(s,t),t}(\eta) \big|\big) \lesssim 2^{-n\varepsilon}|t-s|^{\varepsilon} \big\{ K^{H_{\varepsilon,0}}(\eta)+K^{H_{\varepsilon,1}}(\eta)+K^{H_{\varepsilon,2}}(\eta) \big\} \ ,
\end{equation}
where  $H_{\varepsilon,0} \triangleq (H_0-\varepsilon,H_1,H_2)$, $H_{\varepsilon,1} \triangleq (H_0,H_1-\varepsilon,H_2)$, $H_{\varepsilon,2} \triangleq (H_0,H_1,H_2-\varepsilon)$, and the proportional constants do no depend on $(n,m)$, $(s,t)$ and $\eta$.
\end{lemma}

\begin{proof}
Thanks to \cite[Corollary 2.2]{deya-wave}, we can assert that, under the assumptions of the lemma, and for all  $\rho\geq 0$, one has
\begin{equation}\label{estim-gga}
\int_{\R} d\xi \, \frac{|\ga_{s,t}(\xi,\rho)| |\ga_{u}(\xi,\rho)| }{|\xi|^{2H_0-1}} \lesssim \frac{|t-s|^\varepsilon}{1+\rho^{1+2(H_0-\varepsilon)}} \ ,
\end{equation}
where the proportional constant only depends on $H_0$ and $\varepsilon$. Both estimates (\ref{bou-l-1}) and (\ref{bou-l-2}) immediately follow.
\end{proof}

\

\subsection{Proof of Proposition \ref{prop:stoch-constr}} Following the arguments of \cite[Proposition 1.4]{deya-wave} (which are based on the hypercontractivity property of Gaussian chaos and the classical Garsia-Rodemich-Rumsey lemma), the proof consists in showing that for all $\tau\in \big\{\<Psi>,\<Psi2>,\<IPsi2>,\<PsiIPsi2> \big\}$, $0\leq s<t\leq T$, $1\leq n\leq m$ and $x\in \R^2$, one has
\begin{equation}\label{mom-gene}
\mathbb{E} \bigg[ \Big| \cf^{-1}\Big( \{1+|.|^2\}^{\frac{|\tau|}{2}} \cf \big( \tau_{s,t}^{n,m} \big) \Big)(x)\Big|^2 \bigg] \leq c 2^{-n\varepsilon}|t-s|^\varepsilon \ ,
\end{equation}
for some $\varepsilon >0$, some constant $c>0$ that does not depend on $x$, $n$, $m$, $s$, $t$, and where the \enquote{order} $|\tau|$ of $\tau$ is naturally defined as $|\<Psi>|\triangleq -\al$, $|\<Psi2>|\triangleq -2\al$, $|\<IPsi2>|\triangleq 1-2\al$ and $|\<PsiIPsi2>|\triangleq -\al$.

\subsubsection{Convergence of the first component}
It actually corresponds to the result of \cite[Proposition 1.2]{deya-wave}. Let us only recall that the convergence is here a straightforward consequence of the elementary property
\begin{equation}\label{tech-estim-first-order}
\int_{\R^2} \frac{d\eta}{\{1+|\eta|^2\}^{\al}}  K^H(\eta) \ < \ \infty \ ,
\end{equation}
valid for all $H=(H_0,H_1,H_2)\in (0,1)^3$ and $\al > \frac32-(H_0+H_1+H_2)$. 

\

\subsubsection{Convergence of the second component}\label{subsec:sec-comp}


In this situation, let us first expand the left-hand side of (\ref{mom-gene}) (with $\tau=\<Psi2>$) as
\begin{equation}\label{exp-ordre-deux}
\iint_{(\R^2)^2} \frac{d\la d\lati}{\{1+|\la|^2\}^{\al} \{1+|\lati|^2\}^{\al}}  \,  e^{\imath \langle x,\la-\lati \rangle} \iint_{(\R^2)^2} dy d\yti \, e^{-\imath \langle \la,y \rangle}e^{\imath \langle \lati,\yti\rangle} \mathbb{E}\big[  \<Psi2>^{n,m}_{s,t}(y)\overline{\<Psi2>^{n,m}_{s,t}(\tilde{y})}\big] \ .
\end{equation}
Then, using Wick formula, it is easy to check that the quantity 
$$\mathbb{E}\big[  \<Psi2>^{n,m}_{s,t}(y)\overline{\<Psi2>^{n,m}_{s,t}(\tilde{y})}\big]$$
can be expanded as a sum of terms of the form
\begin{equation}\label{exp-wick}
c_{a,b}\, \mathbb{E}\big[  \<Psi>^{a_1}_{b_1}(y)\overline{\<Psi>^{a_2}_{b_2}(\tilde{y})}\big] \mathbb{E}\big[  \<Psi>^{a_3}_{b_3}(y)\overline{\<Psi>^{a_4}_{b_4}(\tilde{y})}\big]
\end{equation}
where $a_i\in \{n,m,\{n,m\}\}$, $b_i\in \{s,t,\{s,t\}\}$, and one has both $\{a_1,\ldots,a_4\} \cap \{\{n,m\}\}\neq \emptyset$ and $\{b_1,\ldots,b_4\} \cap \{\{s,t\}\}\neq \emptyset$ (in other words, each of the summands contains at least one increment with respect to $(n,m)$ and one increment with respect to $(s,t)$). An element in this set is for instance given by
\begin{equation}\label{example-pr}
\mathbb{E}\big[  \<Psi>^{n,m}_{s,t}(y)\overline{\<Psi>^{m}_{t}(\tilde{y})}\big] \mathbb{E}\big[  \<Psi>^{m}_{t}(y)\overline{\<Psi>^{m}_{t}(\tilde{y})}\big] \ .
\end{equation}
By formula (\ref{cova-gene}), this element can be expanded as
\begin{align*}
&\mathbb{E}\big[  \<Psi>^{n,m}_{s,t}(y)\overline{\<Psi>^{m}_{t}(\tilde{y})}\big] \mathbb{E}\big[  \<Psi>^{m}_{t}(y)\overline{\<Psi>^{m}_{t}(\tilde{y})}\big]\\
&=\iint d\eta d\etati \,  e^{\imath \langle \eta,y\rangle}e^{-\imath \langle \eta,\yti\rangle}  e^{\imath \langle \etati,y\rangle}e^{-\imath \langle \etati,\yti\rangle} L^{H,((n,m),m)}_{(s,t),t}(\eta)L^{H,(m,m)}_{t,t}(\etati) \ ,
\end{align*}
and so, going back to (\ref{exp-ordre-deux}), we get that
\begin{align*}
&\iint_{(\R^2)^2} \frac{d\la d\lati}{\{1+|\la|^2\}^{\al} \{1+|\lati|^2\}^{\al}}  \,  e^{\imath \langle x,\la-\lati \rangle} \iint_{(\R^2)^2} dy d\yti \, e^{-\imath \langle \la,y \rangle}e^{\imath \langle \lati,\yti\rangle} \mathbb{E}\big[  \<Psi>^{n,m}_{s,t}(y)\overline{\<Psi>^{m}_{t}(\tilde{y})}\big] \mathbb{E}\big[  \<Psi>^{m}_{t}(y)\overline{\<Psi>^{m}_{t}(\tilde{y})}\big]\\
&\hspace{2cm}=\iint_{(\R^2)^2}  \frac{d\eta d\etati}{\{1+|\eta-\etati|^2\}^{2\al}}  L^{H,((n,m),m)}_{(s,t),t}(\eta)L^{H,(m,m)}_{t,t}(\etati) \ .
\end{align*}
Now we can use Lemma \ref{lem:bou-l-h} to derive that, for $\varepsilon>0$ small enough,
\begin{align*}
&\bigg| \iint_{(\R^2)^2}\frac{d\la d\lati}{\{1+|\la|^2\}^{\al} \{1+|\lati|^2\}^{\al}}  \,  e^{\imath \langle x,\la-\lati \rangle} \iint_{(\R^2)^2} dy d\yti \, e^{-\imath \langle \la,y \rangle}e^{\imath \langle \lati,\yti\rangle} \mathbb{E}\big[  \<Psi>^{n,m}_{s,t}(y)\overline{\<Psi>^{m}_{t}(\tilde{y})}\big] \mathbb{E}\big[  \<Psi>^{m}_{t}(y)\overline{\<Psi>^{m}_{t}(\tilde{y})}\big] \bigg|\\
&\hspace{2cm}\lesssim 2^{-n\varepsilon} |t-s|^\varepsilon\iint_{(\R^2)^2} \frac{d\eta d\etati}{\{1+|\eta-\etati|^2\}^{2\al}} K^{H_{0,\varepsilon}}(\etati) \big\{ K^{H_{\varepsilon,0}}(\eta)+K^{H_{\varepsilon,1}}(\eta)+K^{H_{\varepsilon,2}}(\eta) \big\} \ .
\end{align*}
At this point, it should be clear to the reader that the above arguments could actually be applied to any element of the form (\ref{exp-wick}), allowing us to assert that for any $\varepsilon>0$ small enough, 
\begin{align}
&\mathbb{E} \bigg[ \Big| \cf^{-1}\Big( \{1+|.|^2\}^{-\al} \cf \big( \<Psi2>_{s,t}^{n,m} \big) \Big)(x)\Big|^2 \bigg]\lesssim 2^{-n\varepsilon} |t-s|^\varepsilon\nonumber\\
&\hspace{0.5cm}\iint_{(\R^2)^2} \frac{d\eta d\etati}{\{1+|\eta-\etati|^2\}^{2\al}}  \big\{ K^{H_{\varepsilon,0}}(\eta)+K^{H_{\varepsilon,1}}(\eta)+K^{H_{\varepsilon,2}}(\eta) \big\} \big\{ K^{H_{\varepsilon,0}}(\etati)+K^{H_{\varepsilon,1}}(\etati)+K^{H_{\varepsilon,2}}(\etati) \big\}\ .\label{conclu-ordre-deux}
\end{align}
The conclusion is then an immediate consequence of the following technical result:

\

\begin{lemma}\label{lem:tech-ordre-deux}
For all $H=(H_0,H_1,H_2), \Hti=(\Hti_0,\Hti_1,\Hti_2) \in (0,1)^3$ satisfying
\begin{equation}\label{constraint-h-i-lem}
0<H_1, \Hti_1<\frac34 \quad , \quad 0<H_2,\Hti_2< \frac34 \quad , \quad 1 < H_0+H_1+H_2 \leq  \frac54 \quad ,\quad 1 < \Hti_0+\Hti_1+\Hti_2 \leq  \frac54  \ ,
\end{equation} 
and any
\begin{equation}\label{tech-cond-al}
\al\in \big(\max\big(\frac32-(H_0+H_1+H_2),\frac32-(\Hti_0+\Hti_1+\Hti_2)\big),\frac12\big) \ ,
\end{equation}
it holds that
$$
\iint_{(\R^2)^2} \frac{d\eta d\etati}{\{1+|\eta-\etati|^2\}^{2\al}} K^H(\eta) K^{\Hti}(\etati) \ < \ \infty\ .
$$
\end{lemma}

\begin{proof}

\

\noindent
\textit{Step 0: Simplification of the problem.} Let us show that the problem actually reduces to the consideration of the four following integrals:
$$\cj_1\triangleq \int_{\R^2} d\eta \int_{\R^2} d\etati\, \frac{1}{\{1+|\eta|^2\}^{\al}} \frac{1}{\{1+|\etati|^2\}^{\al}}K^H(\eta) K^{\Hti}(\etati) \ ,$$
$$
\cj_2 \triangleq \int_0^\infty d\eta_1 \int_{0}^{\infty} d\etati_1\int_0^\infty d\eta_2\int_{\eta_2}^{2\eta_2} d\etati_2\, \frac{1}{\{1+\eta_1^2+(\eta_2-\etati_2)^2\}^{\al}}\frac{1}{\{1+\etati_1^2+(\eta_2-\etati_2)^2\}^{\al}} K^H(\eta) K^{\Hti}(\etati) \ ,
$$
$$
\cj_3 \triangleq \int_0^\infty d\eta_1 \int_{\eta_1}^{2\eta_1} d\etati_1\int_0^\infty d\eta_2\int_{\eta_2}^{2\eta_2} d\etati_2\frac{1}{\{1+|\eta-\etati|^2\}^{2\al}} K^H(\eta) K^{\Hti}(\etati) \ .
$$
and
$$
\cj_4 \triangleq \int_0^\infty d\eta_1 \int_{\eta_1}^{2\eta_1} d\etati_1\int_0^\infty d\etati_2\int_{\etati_2}^{2\etati_2} d\eta_2\frac{1}{\{1+|\eta-\etati|^2\}^{2\al}} K^H(\eta) K^{\Hti}(\etati) \ .
$$

\smallskip

First, observe that for obvious symmetric and sign reasons, we can focus on the integration over the two domains $\cd_1 \triangleq \{\eta_1< 0<\etati_1 \, , \, (\eta_2,\etati_2)\in \R^2\}$ and $\cd_2 \triangleq \{0<\eta_1< \etati_1 \, , \, (\eta_2,\etati_2)\in \R^2\}$. 

\smallskip

As far as $\cd_1$ is concerned, let us decompose the domain as $\cd_1=\cd_1^1 \cup \cd_1^2$, with $\cd_1^1\triangleq \{\eta_1<0<\etati_1 \, , \, n_2\etati_2 <0\}$ and $\cd_1^1\triangleq \{\eta_1<0<\etati_1 \, , \, \eta_2\etati_2 \geq 0\}$. 
 For $(\eta,\etati)\in \cd_1^1$, one has $|\eta_i-\etati_i|^2 \geq \max\big( |\eta_i|^2,|\etati_i|^2\big)$ for $i\in \{1,2\}$, and so the integral over $\cd_1^1$ is bounded by $\cj_1$. For $(\eta,\etati)\in \cd_1^2$, one has again $|\eta_1-\etati_1|^2 \geq \max\big( |\eta_1|^2,|\etati_1|^2\big)$, as well as one of the following four situations: $|\eta_2| \geq 2 |\etati_2|$, $|\etati_2| \geq 2 |\eta_2|$, $|\eta_2| \leq |\etati_2| \leq 2|\eta_2|$ or $|\etati_2|\leq |\eta_2| \leq 2|\etati_2|$. In the first two cases, one has $|\eta_2-\etati_2|^2 \geq \frac14 \max\big( |\eta_2|^2,|\etati_2|^2\big)$ and so we can again go back to the integral $\cj_1$, while the integration in the third and fourth cases clearly reduces to the consideration of $\cj_2$.

\smallskip

Along the same ideas, decompose $\cd_2$ into $\cd_2 =\cd_2^1\cup \cd_2^2$, with $\cd_2^1\triangleq \{0<\eta_1< \etati_1 <2\eta_1 \, , \, (\eta_2,\etati_2)\in \R^2\}$ and $\cd_2^2\triangleq\{0<2\eta_1< \etati_1 \, , \, (\eta_2,\etati_2)\in \R^2\}$, and observe that if $(\eta,\etati)\in \cd_2^2$, then $|\eta_1-\etati_1|^2 \geq \frac14 \max\big( |\eta_1|^2,|\etati_1|^2)$. We can finally use the same splitting as above for $(\eta_2,\etati_2)$ in order to reduce the problem to the consideration of one of the four integrals $\cj_i$, $i\in \{1,\ldots,4\}$.

\

\noindent
\textit{Step 1: Estimation of $\cj_1$.} The quantity under consideration here can of course be written as
\begin{equation}
\cj_1=\bigg(\int_{\R^2} \frac{d\eta}{\{1+|\eta|^2\}^{\al}}  K^H(\eta) \bigg) \bigg(\int_{\R^2} \frac{d\etati}{\{1+|\etati|^2\}^{\al}}  K^{\Hti}(\etati) \bigg) \ ,\label{estim-order-one}\\
\end{equation}
and we can thus conclude with the first-order statement (\ref{tech-estim-first-order}).

\

\noindent
\textit{Step 2: Estimation of $\cj_2$.} One has here 
\begin{eqnarray}
\lefteqn{\cj_2\ =\int_0^\infty d\eta_1 \int_{0}^{\infty} d\etati_1\int_0^\infty d\eta_2 \, \eta_2\int_{0}^{1} dr\, \frac{1}{\{1+\eta_1^2+\eta_2^2r^2\}^{\al}}\frac{1}{\{1+\etati_1^2+\eta_2^2r^2\}^{\al}}}\nonumber \\
& &\hspace{2cm}\frac{|\eta_1|^{1-2H_1} |\etati_1|^{1-2\Hti_1}}{|\eta_2|^{2(H_2+\Hti_2)-2}(1+r)^{2\Hti_2-1}} \frac{1}{\{ 1+|\eta|^{1+2H_0}\}}\frac{1}{\{ 1+(\etati_1^2+\etati_2^2(1+r)^2)^{\frac12+\Hti_0}\}}\nonumber\\
&\lesssim& \int_0^\infty d\eta_1 \int_{0}^{\infty} d\etati_1\int_0^\infty d\eta_2\, \frac{|\eta_1|^{1-2H_1} |\etati_1|^{1-2\Hti_1}}{|\eta_2|^{2(H_2+\Hti_2)-3}}  \frac{1}{\{ 1+|\eta|^{1+2H_0}\}\{ 1+|\etati|^{1+2\Hti_0}\}}\nonumber \\
& &\hspace{5cm} \int_{0}^{1} dr\, \frac{1}{\{1+\eta_1^2+\eta_2^2r^2\}^{\al}}\frac{1}{\{1+\etati_1^2+\eta_2^2r^2\}^{\al}}\nonumber\\
&\lesssim& \int_0^\infty \frac{d\eta_2}{|\eta_2|^{2(H_2+\Hti_2)-3}} \bigg[ \int_0^\infty \frac{d\eta_1}{|\eta_1|^{2H_1-1} \{1+(\eta_1^2+\eta_2^2)^{H_0+\frac12}\}}\bigg( \int_0^1 \frac{dr}{\{1+\eta_1^2+\eta_2^2r^2\}^{2\al}} \bigg)^{\frac12} \bigg]\nonumber\\
& &\hspace{3cm}\bigg[ \int_0^\infty \frac{d\etati_1}{|\etati_1|^{2\Hti_1-1} \{1+(\etati_1^2+\eta_2^2)^{\Hti_0+\frac12}\}}\bigg( \int_0^1 \frac{dr}{\{1+\etati_1^2+\eta_2^2r^2\}^{2\al}} \bigg)^{\frac12} \bigg]\nonumber\\
&\lesssim& \int_0^\infty \frac{d\eta_2}{|\eta_2|^{2(H_2+\Hti_2)-3}} \bigg( \int_0^\infty \frac{d\eta_1}{|\eta_1|^{4H_1-2} \{1+(\eta_1^2+\eta_2^2)^{2H_0+1}\}}\bigg)^{\frac12}\nonumber\\
& &\hspace{2cm} \bigg( \int_0^\infty \frac{d\etati_1}{|\etati_1|^{4\Hti_1-2} \{1+(\etati_1^2+\eta_2^2)^{2\Hti_0+1}\}}\bigg)^{\frac12}\bigg( \int_0^\infty d\la \int_0^1 \frac{dr}{\{1+\la^2+\eta_2^2 r^2\}^{2\al}} \bigg)\nonumber\\
&\lesssim& \int_0^\infty \frac{d\eta_2}{|\eta_2|^{2(H_2+\Hti_2)-2}} \bigg( \int_0^\infty \frac{d\eta_1}{|\eta_1|^{4H_1-2} \{1+(\eta_1^2+\eta_2^2)^{2H_0+1}\}}\bigg)^{\frac12}\nonumber\\
& & \hspace{2cm}\bigg( \int_0^\infty \frac{d\etati_1}{|\etati_1|^{4\Hti_1-2} \{1+(\etati_1^2+\eta_2^2)^{2\Hti_0+1}\}}\bigg)^{\frac12}\bigg( \int_0^\infty d\la \int_0^{\eta_2} \frac{dr}{\{1+\la^2+ r^2\}^{2\al}} \bigg)\nonumber\\
&\lesssim& \bigg(\int_0^\infty \frac{d\eta_2}{|\eta_2|^{4H_2-2}}  \int_0^\infty \frac{d\eta_1}{|\eta_1|^{4H_1-2} \{1+|\eta|^{4H_0+2}\}}\bigg( \int_0^\infty d\la \int_0^{\eta_2} \frac{dr}{\{1+\la^2+ r^2\}^{2\al}} \bigg)\bigg)^{\frac12}\nonumber\\
& &\hspace{1cm}\bigg(\int_0^\infty \frac{d\etati_2}{|\etati_2|^{4\Hti_2-2}}  \int_0^\infty \frac{d\etati_1}{|\etati_1|^{4\Hti_1-2} \{1+|\etati|^{2\Hti_0+1}\}}\bigg( \int_0^\infty d\la \int_0^{\etati_2} \frac{dr}{\{1+\la^2+ r^2\}^{2\al}} \bigg)\bigg)^{\frac12} \ .\label{ref-lem-tech-i-2}
\end{eqnarray}
At this point, let us pick $\varepsilon >0$ such that $2\al-\frac12-\varepsilon >0$ (noting that $\al>\frac14$ by (\ref{constraint-h-i-lem})) and write
$$
 \int_0^\infty d\la \int_0^{\eta_2} \frac{dr}{\{1+\la^2+ r^2\}^{2\al}} \leq  \int_0^\infty \frac{d\la}{\{1+\la^2\}^{\frac12+\varepsilon}}  \int_0^{\eta_2} \frac{dr}{\{1+r^2\}^{2\al-\frac12-\varepsilon}}\lesssim 1+|\eta_2|^{2-4\al+2\varepsilon} \ ,
$$
so that
\begin{align*}
&\int_0^\infty \frac{d\eta_2}{|\eta_2|^{4H_2-2}}  \int_0^\infty \frac{d\eta_1}{|\eta_1|^{4H_1-2} \{1+|\eta|^{4H_0+2}\}}\bigg( \int_0^\infty d\la \int_0^{\eta_2} \frac{dr}{\{1+\la^2+ r^2\}^{2\al}} \bigg)\\
&\lesssim \int_0^\infty \frac{d\eta_1}{|\eta_1|^{4H_1-2} \{1+|\eta_1|^{4H_0+2}\}} \int_0^1 \frac{d\eta_2}{|\eta_2|^{4H_2-2}}+\int_0^\infty d\eta_1 \int_1^\infty d\eta_2 \, \frac{1}{|\eta_1|^{4H_1-2} |\eta_2|^{4H_2+4\al-4-2\varepsilon} |\eta|^{4H_0+2}} \ .\\
\end{align*}
Using conditions (\ref{constraint-h-i-lem})-(\ref{tech-cond-al}), it is easy to check that the latter integrals are finite for any $\varepsilon>0$ small enough. It is then clear that these arguments also apply to the second term in (\ref{ref-lem-tech-i-2}), which achieves to proof the finiteness of $\cj_2$.

\

\noindent
\textit{Step 3: Estimation of $\cj_3$.} Let us write
\begin{eqnarray*}
\cj_3 &=&\int_0^\infty d\eta_1 \int_{0}^{\infty} d\eta_2 \, \eta_1 \eta_2 \int_0^1 dr_1 \int_0^1 dr_2 \frac{1}{\{1+\eta_1^2 r_1^2+\eta_2^2r_2^2\}^{2\al}}\\
& & \frac{(1+r_1)^{1-2\Hti_1}(1+r_2)^{1-2\Hti_2}}{|\eta_1|^{2(H_1+\Hti_1)-2}|\eta_2|^{2(H_2+\Hti_2)-2}}  \frac{1}{1+|\eta|^{2H_0+1}} \frac{1}{1+(\eta_1^2(1+r_1)^2+\eta_2^2(1+r_2)^2)^{\Hti_0+\frac12}}\\
&\lesssim& \int_0^\infty \int_0^\infty \frac{d\eta_1 d\eta_2}{|\eta_1|^{2(H_1+\Hti_1)-2} |\eta_2|^{2(H_2+\Hti_2)-2}\{1+|\eta|^{2(H_0+\Hti_0)+2}\}} \int_0^{\eta_1} dr_1 \int_0^{\eta_2} dr_2 \, \frac{1}{\{1+r_1^2+r_2^2\}^{2\al}}\\
&\lesssim& \int_{0<\eta_1,\eta_2<1} \frac{d\eta_1 d\eta_2}{|\eta_1|^{2(H_1+\Hti_1)-3} |\eta_2|^{2(H_2+\Hti_2)-3}}\\
& & +\int_{\substack{0<\eta_1,\eta_2<\infty\\ \eta_1 \vee \eta_2 >1}}   \frac{d\eta_1 d\eta_2}{|\eta_1|^{2(H_1+\Hti_1)-2} |\eta_2|^{2(H_2+\Hti_2)-2}\{1+|\eta|^{2(H_0+\Hti_0)+2}\}}\iint_{[0,\eta_1 \vee \eta_2]^2} \frac{dr_1 dr_2}{\{1+r_1^2+r_2^2\}^{2\al}} \ .
\end{eqnarray*}
The first integral is clearly finite. Then, since $\al \in (\frac14,\frac12)$,
$$\iint_{[0,\eta_1 \vee \eta_2]^2} \frac{dr_1 dr_2}{\{1+r_1^2+r_2^2\}^{2\al}} \lesssim \int_0^{2(\eta_1 \vee \eta_2)} d\rho \, \frac{\rho}{\{1+\rho^2\}^{2\al}} \lesssim 1+(\eta_1 \vee \eta_2)^{2-4\al} \ ,$$
and so, using the fact that $H_i,\Hti_i \in (0,\frac34)$ for $i\in \{1,2\}$, we get
\begin{align*}
&\int_{\substack{0<\eta_1,\eta_2<\infty\\ \eta_1 \vee \eta_2 >1}}   \frac{d\eta_1 d\eta_2}{|\eta_1|^{2(H_1+\Hti_1)-2} |\eta_2|^{2(H_2+\Hti_2)-2}\{1+|\eta|^{2(H_0+\Hti_0)+2}\}}\iint_{[0,\eta_1 \vee \eta_2]^2} \frac{dr_1 dr_2}{\{1+r_1^2+r_2^2\}^{2\al}}\\
&\hspace{6cm}\lesssim \int_1^\infty \frac{dr}{r^{2(H_0+H_1+H_2)+2(\Hti_0+\Hti_1+\Hti_2)+4\al-5}} \ .
\end{align*}
Thanks to (\ref{tech-cond-al}), we can assert that the latter integral is finite, and accordingly $\cj_3$ is finite too.

\

\noindent
\textit{Step 4: Estimation of $\cj_4$.} We have of course
$$
\cj_4 =\int_0^\infty d\eta_1 \int_{\eta_1}^{2\eta_1} d\etati_1\int_0^\infty d\eta_2\int_{\frac{\eta_2}{2}}^{\eta_2} d\etati_2\frac{1}{\{1+|\eta-\etati|^2\}^{2\al}} K^H(\eta) K^H(\etati) \ ,
$$
and from there it is easy to mimic the arguments that we have used for $\cj_3$.

\end{proof}

\

\subsubsection{Convergence of the third component}

Noting that
$$\fouri\big( \<IPsi2>^{n,m}_{s,t}\big)(\la)=\fouri\big( \big( G\ast \<Psi2>^{n,m}\big)_{s,t}\big)(\la)=\int_0^t du \, \fouri(G_u)(\la) \fouri\big(\<Psi2>^{n,m}_{s-u,t-u}\big)(\la) \ ,$$
we get
\begin{align*}
&\mathbb{E} \bigg[ \Big| \cf^{-1}\Big( \{1+|.|^2\}^{\frac12 (1-2\al)} \cf \big( \<IPsi2>^{n,m}_{s,t}\big) \big) \Big)(x)\Big|^2 \bigg]\\
&=\iint_{(\R^2)^2} \frac{d\la d\lati}{\{1+|\la|^2\}^{\frac12 (2\al-1)} \{1+|\lati|^2\}^{\frac12 (2\al-1)}}  \iint_{[0,t]^2} du d\uti\,  e^{\imath \langle x,\la-\lati \rangle} \fouri(G_u)(\la) \overline{\fouri(G_{\uti})(\lati)}\\
& \hspace{6cm} \iint_{(\R^2)^2} dy d\yti \, e^{-\imath \langle \la,y \rangle}e^{\imath \langle \lati,\yti\rangle} \mathbb{E}\big[  \<Psi2>^{n,m}_{s-u,t-u}(y)\overline{\<Psi2>^{n,m}_{s-\uti,t-\uti}(\tilde{y})}\big] \ .
\end{align*}
Starting from this expression, we can easily follow the lines of the above reasoning (for the second component) and derive that for any $\varepsilon >0$ small enough,
\begin{align*}
&\mathbb{E} \bigg[ \Big| \cf^{-1}\Big( \{1+|.|^2\}^{\frac12 (1-2\al)} \cf \big( \<IPsi2>^{n,m}_{s,t} \big) \Big)(x)\Big|^2 \bigg]\\
&\lesssim 2^{-n\varepsilon} |t-s|^\varepsilon\iint_{(\R^2)^2} \frac{d\eta d\etati}{\{1+|\eta-\etati|^2\}^{2\al-1}}\iint_{[0,t]^2} du d\uti\, |\fouri(G_u)(\eta-\etati)| |\fouri(G_{\uti})(\eta-\etati)|\\
&\hspace{3cm}  \big\{ K^{H_{\varepsilon,0}}(\eta)+K^{H_{\varepsilon,1}}(\eta)+K^{H_{\varepsilon,2}}(\eta) \big\} \big\{ K^{H_{\varepsilon,0}}(\etati)+K^{H_{\varepsilon,1}}(\etati)+K^{H_{\varepsilon,2}}(\etati) \big\}\ .
\end{align*}
Using the elementary estimate
\begin{equation}\label{bou-fouri-g}
\sup_{u\in [0,1]}\big| \cf\big( G_{u}\big)(\eta-\etati) \big| \lesssim \{1+|\eta-\etati|^2\}^{-\frac12} \ ,
\end{equation}
we have thus, for any $\varepsilon >0$ small enough,
\begin{align*}
&\mathbb{E} \bigg[ \Big| \cf^{-1}\Big( \{1+|.|^2\}^{\frac12 (1-2\al)} \cf \big( \<IPsi2>^{n,m}_{s,t}  \big) \Big)(x)\Big|^2 \bigg]\lesssim 2^{-n\varepsilon} |t-s|^\varepsilon\\
&\iint_{(\R^2)^2} \frac{d\eta d\etati}{\{1+|\eta-\etati|^2\}^{2\al}}  \big\{ K^{H_{\varepsilon,0}}(\eta)+K^{H_{\varepsilon,1}}(\eta)+K^{H_{\varepsilon,2}}(\eta) \big\} \big\{ K^{H_{\varepsilon,0}}(\etati)+K^{H_{\varepsilon,1}}(\etati)+K^{H_{\varepsilon,2}}(\etati) \big\} \ .
\end{align*}
Observe that we are here in the very same position as in (\ref{conclu-ordre-deux}), and so, using the same technical Lemma \ref{lem:tech-ordre-deux}, we get the desired estimate (\ref{mom-gene}) for $\tau=\<IPsi2>$.

\

\subsubsection{Convergence of the fourth component}

First, observe that $\fouri \big(\<PsiIPsi2>^{n,m}_{s,t}\big) $ can be readily expanded as
\begin{eqnarray*}
\fouri \big(\<PsiIPsi2>^{n,m}_{s,t} \big) &=&\fouri \big( \big( G\ast \<Psi2>^{n,m}\big)_{s,t}\big)  \ast \fouri\big( \<Psi>^m_t \big)+\fouri \big( \big( G\ast \<Psi2>^{n}\big)_{s,t}\big)  \ast \fouri\big( \<Psi>^{n,m}_t \big)\\
& &+\fouri \big( \big( G\ast \<Psi2>^{n,m}\big)_{s}\big)  \ast \fouri\big( \<Psi>^m_{s,t} \big)+\fouri \big( \big( G\ast \<Psi2>^{n}\big)_{s}\big)  \ast \fouri\big( \<Psi>^{n,m}_{s,t} \big) \ .
\end{eqnarray*}
As it should be clear to the reader, the subsequent arguments could be applied to any of these four terms, and thus we will only focus on the estimate of
\begin{align}
&\mathbb{E} \bigg[ \Big| \cf^{-1}\Big( \{1+|.|^2\}^{-\frac{\al}{2}} \big(\cf \big( ( G \ast \<Psi2>^{n,m})_{s,t}\big) \ast \cf\big(\<Psi>^m_t\big) \big)\Big)(x)\Big|^2 \bigg]\nonumber \\
&=\iint_{(\R^2)^2} \frac{d\la d\lati}{\{1+|\la|^2\}^{\frac{\al}{2}} \{1+|\lati|^2\}^{\frac{\al}{2}}}  \iint_{(\R^2)^2} d\beta  d\betati\,  e^{\imath \langle x,\la-\lati \rangle} \iint_{[0,t]^2} du d\uti\, \fouri\big( G_{u}\big)(\beta) \overline{\fouri\big( G_{\uti}\big)(\betati)}\nonumber \\
& \qquad \iint_{(\R^2)^2} dy d\yti \iint_{(\R^2)^2} dz d\zti \, e^{-\imath \langle \beta,y \rangle}e^{\imath \langle \betati,\yti\rangle} e^{-\imath \langle \la-\beta,z\rangle}  e^{\imath \langle \lati-\betati,\zti\rangle} \mathbb{E} \Big[ \<Psi2>^{n,m}_{s-u,t-u}(y)\<Psi>^m_t(z) \overline{\<Psi2>^{n,m}_{s-\uti,t-\uti}(\yti)} \overline{\<Psi>^m_t(\zti)}  \Big] \ .\label{order-three-1}
\end{align}
Using Wick formula, we can then expand (along the same idea as in Section \ref{subsec:sec-comp}) the expectation
$$\mathbb{E} \Big[ \<Psi2>^{n,m}_{s-u,t-u}(y)\<Psi>^m_t(z) \overline{\<Psi2>^{n,m}_{s-\uti,t-\uti}(\yti)} \overline{\<Psi>^m_t(\zti)}  \Big] \,$$
as a sum of terms of the form
\begin{equation}\label{wick-term-1}
c_{a;b}^1\, \mathbb{E} \big[ \<Psi>^{a_1}_{b_1}(y) \overline{\<Psi>^{a_2}_{b_2}(\yti)}\big] \mathbb{E} \big[ \<Psi>^{a_3}_{b_3}(y) \overline{\<Psi>^{a_4}_{b_4}(\yti)}\big]   \mathbb{E} \big[ \<Psi>^m_t(z) \overline{\<Psi>^m_t(\zti)} \big]\ ,
\end{equation}
\begin{equation}\label{wick-term-2}
c_{a;b}^2\, \mathbb{E} \big[ \<Psi>^{a_1}_{b_1}(y) \overline{\<Psi>^{a_2}_{b_2}(\yti)}\big] \mathbb{E} \big[ \<Psi>^{a_3}_{b_3}(y) \<Psi>^m_t(z)\big]\mathbb{E} \big[ \overline{\<Psi>^{a_4}_{b_4}(\yti)} \overline{\<Psi>^m_t(\zti)}\big]
\end{equation}
or
\begin{equation}\label{wick-term-3}
c_{a;b}^3\, \mathbb{E} \big[ \<Psi>^{a_1}_{b_1}(y) \overline{\<Psi>^{a_2}_{b_2}(\yti)}\big] \mathbb{E} \big[ \<Psi>^{a_3}_{b_3}(y) \overline{\<Psi>^m_t(\zti)}\big] \mathbb{E} \big[ \overline{\<Psi>^{a_4}_{b_4}(\yti)} \<Psi>^m_t(z)\big] \ , 
\end{equation}
where $a_i\in \{n,m,\{n,m\}\}$, $b_1,b_3\in \{s-u,t-u,\{s-u,t-u\}\}$, $b_2,b_4\in \{s-\uti,t-\uti,\{s-\uti,t-\uti\}\}$, and one has both $\{a_1,\ldots,a_4\} \cap \{\{n,m\}\} \neq \emptyset$ and
$\{b_1,\ldots,b_4\} \cap \big( \{\{s-u,t-u\},\{s-\uti,t-\uti\}\} \big) \neq \emptyset$. An example of a pair $(a;b)$ satisfying these conditions is given by
\begin{equation}\label{examp-a-b}
(a;b)=\big((\{n,m\},m,m,m)\, ;\, (\{s-u,t-u\},t-\uti,t-u,t-\uti)\big) \ .
\end{equation}
In the sequel, and for the sake of clarity, we will only focus on the estimates associated with this particular pair $(a;b)$, but it should (again) be clear to the reader that any other pair $(a;b)$ satisfying the above conditions could be treated with similar arguments.   

\smallskip

Injecting successively (\ref{wick-term-1}), (\ref{wick-term-2}) and (\ref{wick-term-3}) into (\ref{order-three-1}) (with $(a;b)$ fixed as in (\ref{examp-a-b})) gives rise to the consideration of three specific integrals, that we denote by $\cj_1$, $\cj_2$ and $\cj_3$, respectively.

\

\noindent
\textit{Estimation of $\cj_1$.} Using the covariance formula (\ref{cova-gene}), we get on the one hand 
\begin{align}
&\iint_{(\R^2)^2} dy d\yti \, e^{-\imath \langle \beta,y \rangle}e^{\imath \langle \betati,\yti\rangle} \mathbb{E} \big[ \<Psi>^{n,m}_{s-u,t-u}(y) \overline{\<Psi>^{m}_{t-\uti}(\yti)}\big] \mathbb{E} \big[ \<Psi>^{m}_{t-u}(y) \overline{\<Psi>^{m}_{t-\uti}(\yti)}\big]  \nonumber \\
&=c \iint_{(\R^2)^2} d\eta d\etati \, L^{H,((n,m),m)}_{(s-u,t-u),t-\uti}(\eta) L^{H,(m,m)}_{t-u,t-\uti}(\etati) \bigg(\int_{\R^2} dy \, e^{-\imath \langle y,\beta-(\eta+\etati) \rangle}\bigg) \bigg( \int_{\R^2} d\yti \, e^{\imath \langle \yti,\betati-(\eta+\etati)\rangle} \bigg)\nonumber\\
&=c\iint_{(\R^2)^2} d\eta d\etati \, L^{H,((n,m),m)}_{(s-u,t-u),t-\uti}(\eta) L^{H,(m,m)}_{t-u,t-\uti}(\etati) \delta_{\{\beta=\betati=\eta+\etati\}} \ ,\label{fouri-simpli}
\end{align}
and on the other hand
$$\iint_{(\R^2)^2} dz d\zti \, e^{-\imath \langle \la-\beta,z\rangle}  e^{\imath \langle \lati-\betati,\zti\rangle} \mathbb{E} \big[ \<Psi>^{m}_{t}(z) \overline{\<Psi>^{m}_{t}(\zti)} \big]= c\int_{\R^2} d\etatiti \, L^{H,(m,m)}_{t,t}(\etatiti) \, \delta_{\{\la-\beta=\lati-\betati=\etatiti\}} \  ,$$
so that
\begin{align}
&\cj_1=c\iint_{[0,t]^2} du d\uti \iiint_{(\R^2)^3} \frac{d\eta d\etati d\etatiti}{\{1+|\eta+\etati+\etatiti|^2\}^{\al}} \,  L^{H,((n,m),m)}_{(s-u,t-u),t-\uti}(\eta) L^{H,(m,m)}_{t-u,t-\uti}(\etati) L^{H,(m,m)}_{t,t}(\etatiti)\nonumber\\
&\hspace{6cm} \fouri\big( G_{u}\big)(\eta+\etati) \overline{\fouri\big( G_{\uti}\big)(\eta+\etati)}\label{ref-i-1} \\
&=c\iint_{[0,t]^2} du d\uti \iint_{(\R^2)^2} d\eta d\etatiti \, \{1+|\eta+\etatiti|^2\}^{-\al}\fouri\big( G_{u}\big)(\eta) \overline{\fouri\big( G_{\uti}\big)(\eta)}\nonumber\\
&  \hspace{3cm} L^{H,(m,m)}_{t,t}(\etatiti)\bigg( \int_{\R^2} d\etati \, L^{H,((n,m),m)}_{(s-u,t-u),t-\uti}(\eta-\etati)L^{H,(m,m)}_{t-u,t-\uti}(\etati) \bigg) \ .\label{ref-i-2}
\end{align}
At this point, we can apply Lemma \ref{lem:bou-l-h} to assert that for any $\varepsilon >0$ small enough,
\begin{align*}
&\bigg| L^{H,(m,m)}_{t,t}(\etatiti)\bigg( \int_{\R^2} d\etati \, L^{H,((n,m),m)}_{(s-u,t-u),t-\uti}(\eta-\etati)L^{H,(m,m)}_{t-u,t-\uti}(\etati) \bigg)\bigg|\\
&\lesssim 2^{-n\varepsilon} |t-s|^\varepsilon K^{H_{\varepsilon,0}}(\etatiti)\sum_{i=0}^2 \bigg( \int_{\R^2} d\etati \, K^{H_{\varepsilon,i}}(\eta-\etati)K^{H_{\varepsilon,0}}(\etati) \bigg) \ .
\end{align*}
Thanks to the forthcoming Lemma \ref{lem:norm-l2-l-h}, we know that for every $i\in \{0,1,2\}$ and for any $\varepsilon >0$ small enough,
$$\sup_{\eta\in \R^2} \int_{\R^2} d\etati \, K^{H_{\varepsilon,i}}(\eta-\etati)K^{H_{\varepsilon,0}}(\etati) \leq  \bigg(\int_{\R^2} d\etati \, \big| K^{H_{\varepsilon,i}}(\etati)\big|^2\bigg)^{\frac12}\bigg(\int d\etati \, \big| K^{H_{\varepsilon,0}}(\etati)\big|^2\bigg)^{\frac12} \ <  \ \infty \ .$$
Going back to (\ref{ref-i-2}) and using also (\ref{bou-fouri-g}), we have thus obtained that 
\begin{equation}\label{bou-i-1}
\cj_1\lesssim 2^{-n\varepsilon} |t-s|^\varepsilon\int_{\R^2} d\etatiti \, K^{H_{\varepsilon,0}}(\etatiti) \int_{\R^2} d\eta \, \{1+|\eta+\etatiti|^2\}^{-\al}\{1+|\eta|^2\}^{-1} \ .
\end{equation}
Applying the subsequent technical Lemma \ref{lem:techn-lem-order-three} finally yields 
$$\cj_1 \lesssim 2^{-n\varepsilon} |t-s|^\varepsilon\int_{\R^2} \frac{d\etatiti}{\{1+|\etatiti|^2\}^{\al-\varepsilon} } \,  K^{H_{\varepsilon,0}}(\etatiti)\ ,$$
and the conclusion now comes from the first-order assertion (\ref{tech-estim-first-order}).

\

\noindent
\textit{Estimation of $\cj_2$.} Using the same arguments as in (\ref{fouri-simpli}) (with the help of both (\ref{cova-gene}) and (\ref{cova-gene-bis})), we get
\begin{align*}
&\iint_{(\R^2)^2} dy d\yti\iint_{(\R^2)^2} dz d\zti  \, e^{-\imath \langle \beta,y \rangle}e^{\imath \langle \betati,\yti\rangle}e^{-\imath \langle \la-\beta,z\rangle}  e^{\imath \langle \lati-\betati,\zti\rangle}\\
&\hspace{3cm} \mathbb{E} \big[ \<Psi>^{n,m}_{s-u,t-u}(y) \overline{\<Psi>^m_{t-\uti}(\yti)}\big] \mathbb{E} \big[ \<Psi>^m_{t-u}(y) \<Psi>^m_{t}(z)\big]\mathbb{E} \big[ \overline{\<Psi>^m_{t-\uti}(\yti)} \overline{\<Psi>^m_{t}(\zti)}\big] \\
&=c\iiint_{(\R^2)^3} d\eta d\etati d\etatiti \, L^{H,((n,m),m)}_{(s-u,t-u),t-\uti}(\eta) \tilde{L}^{H,(m,m)}_{t-u,t}(\etati) \overline{\tilde{L}^{H,(m,m)}_{t-\uti,t}(\etatiti)} \delta_{\{\beta=\eta+\etati\}}\delta_{\{\betati=\eta+\etatiti\}}\delta_{\{\la=\eta+2\etati\}}\delta_{\{\lati=\eta+2\etatiti\}} \ ,
\end{align*}
and so
\begin{align}
&\cj_2=c\iint_{[0,t]^2} du d\uti\iiint_{(\R^2)^3} d\eta d\etati d\etatiti \, L^{H,((n,m),m)}_{(s-u,t-u),t-\uti}(\eta) \tilde{L}^{H,(m,m)}_{t-u,t}(\etati) \overline{\tilde{L}^{H,(m,m)}_{t-\uti,t}(\etatiti)}\nonumber\\
&\hspace{3cm} \{1+|\eta+2\etati|^2\}^{-\frac{\al}{2}} \{1+|\eta+2\etatiti|^2\}^{-\frac{\al}{2}}\fouri\big( G_{u}\big)(\eta+\etati) \overline{\fouri\big( G_{\uti}\big)(\eta+\etatiti)}\nonumber\\
&=c\iint_{[0,t]^2} du d\uti\int_{\R^2} d\eta\, L^{H,((n,m),m)}_{(s-u,t-u),t-\uti}(\eta) \bigg( \int_{\R^2} d\etati \, \tilde{L}^{H,(m,m)}_{t-u,t}(\etati) \{1+|\eta+2\etati|^2\}^{-\frac{\al}{2}}  \fouri\big( G_{u}\big)(\eta+\etati)  \bigg)\nonumber\\
& \hspace{3cm}\bigg( \int_{\R^2} d\etatiti \, \overline{\tilde{L}^{H,(m,m)}_{t-\uti,t}(\etatiti)}\{1+|\eta+2\etatiti|^2\}^{-\frac{\al}{2}}  \overline{\fouri\big( G_{\uti}\big)(\eta+\etatiti)}  \bigg) \ .\label{bou-i-2-1}
\end{align}
Combining (\ref{bou-l-1}) with the result of Lemma \ref{lem:norm-l2-l-h} and estimate (\ref{bou-fouri-g}), we can assert that for any $\varepsilon >0$ small enough,
\begin{eqnarray*}
\lefteqn{ \sup_{u\in [0,t]} \bigg|\int_{\R^2} d\etati \, \tilde{L}^{H,(m,m)}_{t-u,t}(\etati) \{1+|\eta+2\etati|^2\}^{-\frac{\al}{2}}  \fouri\big( G_{u}\big)(\eta+\etati)\bigg|}\\
&\lesssim& \bigg( \int_{\R^2} d\etati \, \big|K^{H_{\varepsilon,0}}(\etati) \big|^2 \bigg)^{1/2} \bigg( \int_{\R^2} d\etati \, \{1+|\eta+2\etati|^2\}^{-\al} \{1+|\eta+\etati|^2\}^{-1} \bigg)^{1/2} \\
&\lesssim&\bigg( \int_{\R^2} d\etati \, \{1+|\eta+\etati|^2\}^{-\al} \{1+|\etati|^2\}^{-1} \bigg)^{1/2} \ ,
\end{eqnarray*}
and similarly
$$\sup_{\uti\in [0,t]} \bigg| \int_{\R^2} d\etatiti \, \overline{\tilde{L}^{H,(m,m)}_{t-\uti,t}(\etatiti)}\{1+|\eta+2\etatiti|^2\}^{-\frac{\al}{2}}  \overline{\fouri\big( G_{t-\uti}\big)(\eta+\etatiti)} \bigg|\lesssim \bigg( \int_{\R^2} d\etatiti \, \{1+|\eta+\etatiti|^2\}^{-\al} \{1+|\etatiti|^2\}^{-1} \bigg)^{1/2}\ ,$$
which, going back to (\ref{bou-i-2-1}), yields
\begin{eqnarray*}
\cj_2 &\lesssim & \iint_{[0,t]^2} du d\uti\int_{\R^2} d\eta\, L^{H,((n,m),m)}_{(s-u,t-u),t-\uti}(\eta) \int_{\R^2} d\etati \, \{1+|\eta+\etati|^2\}^{-\al} \{1+|\etati|^2\}^{-1} \\
&\lesssim & 2^{-n\varepsilon}|t-s|^\varepsilon \sum_{i=0}^2 \int_{\R^2} d\eta\, K^{H_{\varepsilon,i}}(\eta) \int_{\R^2} d\etati \, \{1+|\eta+\etati|^2\}^{-\al} \{1+|\etati|^2\}^{-1} \ .
\end{eqnarray*}
We can then conclude with the same arguments as in (\ref{bou-i-1}).

\

\noindent
\textit{Estimation of $\cj_3$.} As above,
\begin{align*}
&\iint_{(\R^2)^4} dy d\yti dz d\zti  \, e^{-\imath \langle \beta,y \rangle}e^{\imath \langle \betati,\yti\rangle}e^{-\imath \langle \la-\beta,z\rangle}  e^{\imath \langle \lati-\betati,\zti\rangle} \mathbb{E} \big[ \<Psi>^{n,m}_{s-u,t-u}(y) \overline{\<Psi>^m_{t-\uti}(\yti)}\big] \mathbb{E} \big[ \<Psi>^m_{t-u}(y) \overline{\<Psi>^m_{t}(\zti)}\big] \mathbb{E} \big[ \overline{\<Psi>^m_{t-\uti}(\yti)} \<Psi>^m_{t}(z)\big] \\
&\hspace{1cm}=c\iiint_{(\R^2)^3} d\eta d\etati d\etatiti \, L^{H,((n,m),m)}_{(s-u,t-u),t-\uti}(\eta) L^{H,(m,m)}_{t-u,t}(\etati) \overline{L^{H,(m,m)}_{t-\uti,t}(\etatiti)}\, \delta_{\{\beta=\eta+\etati\}}\delta_{\{\betati=\eta+\etatiti\}}\delta_{\{\la=\lati=\eta+\etati+\etatiti\}} \ ,
\end{align*}
and thus, for any $\varepsilon >0$ small enough,
\begin{eqnarray*}
\big| \cj_3\big|&=&c\, \bigg|\iint_{[0,t]^2} du d\uti\iiint_{(\R^2)^3} \frac{d\eta d\etati d\etatiti}{\{1+|\eta+\etati+\etatiti|^2\}^{\al}} \, L^{H,((n,m),m)}_{(s-u,t-u),t-\uti}(\eta) L^{H,(m,m)}_{t-u,t}(\etati) \overline{L^{H,(m,m)}_{t-\uti,t}(\etatiti)}\\
& &\hspace{5cm}  \fouri\big( G_{u}\big)(\eta+\etati) \overline{\fouri\big( G_{\uti}\big)(\eta+\etatiti)}\bigg|\\
&\lesssim&2^{-n\varepsilon} |t-s|^{\varepsilon}\sum_{i=0}^2 \iiint_{(\R^2)^3} d\eta d\etati d\etatiti \, K^{H_{\varepsilon,i}}(\eta) K^{H_{\varepsilon,0}}(\etati) K^{H_{\varepsilon,0}}(\etatiti)\\
& &\hspace{3cm}  \{1+|\eta+\etati+\etatiti|^2\}^{-\al}\{1+|\eta+\etati|^2\}^{-\frac12}\{1+|\eta+\etatiti|^2\}^{-\frac12} \ .
\end{eqnarray*}
Now split the integration domain into $\cd_1\triangleq \{(\eta,\etati,\etatiti): \, \{1+|\eta+\etati|^2\}^{-\frac12} \leq \{1+|\eta+\etatiti|^2\}^{-\frac12}\}$ and $\cd_2 \triangleq \{(\eta,\etati,\etatiti): \, \{1+|\eta+\etatiti|^2\}^{-\frac12} \leq \{1+|\eta+\etati|^2\}^{-\frac12}\}$, and write (trivially)
\begin{align*}
&\iiint_{\cd_1} d\eta d\etati d\etatiti \, K^{H_{\varepsilon,i}}(\eta) K^{H_{\varepsilon,0}}(\etati) K^{H_{\varepsilon,0}}(\etatiti)\{1+|\eta+\etati+\etatiti|^2\}^{-\al}\{1+|\eta+\etati|^2\}^{-\frac12}\{1+|\eta+\etatiti|^2\}^{-\frac12}\\
&\leq \iiint_{(\R^2)^3} d\eta d\etati d\etatiti \, K^{H_{\varepsilon,i}}(\eta) K^{H_{\varepsilon,0}}(\etati) K^{H_{\varepsilon,0}}(\etatiti)\{1+|\eta+\etati+\etatiti|^2\}^{-\al}\{1+|\eta+\etatiti|^2\}^{-1} \ ,
\end{align*}
which essentially brings us back to the integral involved in (\ref{ref-i-1}). We can thus rely on the same arguments as with $\cj_1$ to handle the integral over $\cd_1$. Finally, it is readily checked that these arguments can also be used for the integral over $\cd_2$, which concludes the proof.

\

It only remains us to prove the two technical lemmas at the core of the above reasoning.
\begin{lemma}\label{lem:norm-l2-l-h}
For all $(H_0,H_1,H_2)\in (0,1)^{3}$ such that 
\begin{equation}\label{co-h-lem}
0<H_1<\frac34 \quad , \quad 0<H_2< \frac34 \quad \text{and} \quad H_0+H_1+H_2 >1 \ ,
\end{equation}
it holds that
$$\int_{\R^2} d\eta \, | K^H(\eta)|^2 \ < \ \infty \ .$$ 
\end{lemma}

\begin{proof}
One has, since $4H_1-2<1$ and $4H_2-2<1$, 
$$
\int_{\R^2} d\eta \, | K^H(\eta)|^2  \lesssim \iint_{\R^2} \frac{d\eta_1 d\eta_2}{|\eta_1|^{4H_1-2} |\eta_2|^{4H_2-2}\{1+|\eta|^{2+4H_0}\}}  \lesssim \int_0^\infty \frac{d\rho}{\rho^{4(H_1+H_2)-5}\{1+\rho^{2+4H_0}\} }\ ,
$$
and we can easily check (using (\ref{co-h-lem})) that the latter integral is indeed finite.
\end{proof}

\begin{lemma}\label{lem:techn-lem-order-three}
For all $0<\varepsilon < \al<\frac12$, it holds that
$$\int_{\R^2} d\etati \, \{1+|\eta+\etati|^2\}^{-\al}\{1+|\etati|^2\}^{-1}\lesssim \{1+|\eta|^2\}^{-(\al-\varepsilon)} \ .$$
\end{lemma}

\begin{proof}
Let us first write
$$
\int_{\R^2} d\etati \, \{1+|\eta+\etati|^2\}^{-\al}\{1+|\etati|^2\}^{-1}
 \lesssim \int_{\R^2} d\etati \, \{1+||\eta|-|\etati||^2\}^{-\al}\{1+|\etati|^2\}^{-1}\lesssim \int_0^\infty  \frac{d\rho\, \rho}{1+\rho^2} \{1+||\eta|-\rho|^2\}^{-\al} \ .\\
$$
Now split the integration domain into $\cd_1\triangleq [\frac{|\eta|}{2},\frac{3|\eta|}{2}]$ and $\cd_2 \triangleq \{ 0\leq \rho\leq \frac{|\eta|}{2} \ \text{or} \ \rho\geq \frac{3|\eta|}{2} \}$. On the one hand,
\begin{eqnarray*}
\int_{\cd_1} \frac{d\rho\, \rho}{1+\rho^2} \{1+||\eta|-\rho|^2\}^{-\al}&=&|\eta|^2 \int_{\frac12}^{\frac32} \frac{dr\, r}{\{1+|\eta|^2(1-r)^2\}^\al \{1+|\eta|^2 r^2\}} \\
&\lesssim& \int_{-\frac12}^{\frac12} \frac{dr}{\{1+|\eta|^2 r^2\}^\al} \ \lesssim \ \max\bigg(1,\frac{1}{|\eta|^{2\al}} \int_0^1 \frac{dr}{r^{2\al}} \bigg)  \ \lesssim \ \frac{1}{1+|\eta|^{2\al}} \ .
\end{eqnarray*}
On the other hand, for every $\rho\in \cd_2$, one has $||\eta|-\rho| \geq \frac13 \max(\rho,|\eta|)$, and accordingly
$$\int_{\cd_2} \frac{d\rho\, \rho}{1+\rho^2} \{1+||\eta|-\rho|^2\}^{-\al} \lesssim \{1+|\eta|^2\}^{-(\al-\varepsilon)} \int_0^\infty \frac{d\rho\, \rho}{\{1+\rho^2\}^{1+\varepsilon}}\lesssim \{1+|\eta|^2\}^{-(\al-\varepsilon)} \ .$$
\end{proof}

\subsection{Proof of Proposition \ref{prop:limit-case}}
In the sequel, we use the notation $A\gtrsim B$ whenever there exists a constant $c>0$ such that $A\geq c B$. Besides, without loss of generality, we can here assume that $\al > \frac12 $. For the sake of clarity, let us also introduce the additional notation
\begin{equation}
\gga^{H_0,n}_{t}(\rho)\triangleq\int_{-2^n}^{2^n} d\xi \, \frac{|\ga_t(\xi,\rho)|^2 }{|\xi|^{2H_0-1}} \ .
\end{equation}

\smallskip

\noindent
Using (\ref{cova-gene}) and then Wick formula just as in Section \ref{subsec:sec-comp}, we get that
\begin{eqnarray*}
\lefteqn{\mathbb{E}\Big[ \|\<Psi2>^n(t,.)\|_{\mathcal{W}^{-2\al,2}(D)}^2\Big]} \\
&=&c\int_{|\eta|\leq 2^n} d\eta \int_{|\etati|\leq 2^n} d\etati\, \frac{1}{\{1+|\eta-\etati|^2\}^{2\al}} \frac{\gga^{H_0,n}_{t}(|\eta|)}{|\eta_1|^{2H_1-1} |\eta_2|^{2H_2-1}} \frac{\gga^{H_0,n}_{t}(|\etati|)}{|\etati_1|^{2H_1-1}|\etati_2|^{2H_2-1}}  \\
&\gtrsim&\int_{0}^{2^{n-1}} d\eta_1 \int_{\frac12 \eta_1}^{\eta_1} d\etati_1 \int_0^{2^{n-1}} d\eta_2 \int_{\frac12 \eta_2}^{\eta_2} d\etati_2 \,   \frac{1}{\{1+|\eta-\etati|^2\}^{2\al}}  \frac{\gga^{H_0,n}_{t}(|\eta|)}{|\eta_1|^{2H_1-1} |\eta_2|^{2H_2-1}} \frac{\gga^{H_0,n}_{t}(|\etati|)}{|\etati_1|^{2H_1-1}|\etati_2|^{2H_2-1}} \\
&\gtrsim& \int_{0}^{2^{n-1}} \int_0^{2^{n-1}} \frac{d\eta_1 d\eta_2}{|\eta_1|^{4H_1-3} |\eta_2|^{4H_2-3}}\int_0^{\frac12} \int_0^{\frac12}\frac{dr_1 dr_2}{\{1+\eta_1^2r_1^2+\eta_2^2r_2^2\}^{2\al}} \\
& &\hspace{6cm} \gga^{H_0,n}_{t}(|\eta|) \gga^{H_0,n}_{t} \Big( \sqrt{\eta_1^2 (1-r_1)^2+\eta_2^2(1-r_2)^2}\Big)\\
&\gtrsim& \int_{0<\eta_1<\eta_2<2^{n-1}} \frac{d\eta_1 d\eta_2}{|\eta_1|^{4H_1-2} |\eta_2|^{4H_2-2}}\int_0^{\frac12 \eta_1} \int_0^{\frac12 \eta_1}\frac{dr_1 dr_2}{\{1+r_1^2+r_2^2\}^{2\al}} \\
& &\hspace{6cm} \gga^{H_0,n}_{t}(|\eta|) \gga^{H_0,n}_{t} \Big( \sqrt{\eta_1^2 \big(1-\frac{r_1}{\eta_1}\big)^2+\eta_2^2\big(1-\frac{r_2}{\eta_2}\big)^2}\Big)  \\
&\gtrsim& \int_{\frac{\pi}{8}}^{\frac{\pi}{4}} d\theta \int_2^{2^{n-1}} \frac{d\rho}{\rho^{4H'_1+4H'_2-5}} \int_0^{\frac12 \rho \sin \theta } \int_0^{\frac12 \rho \sin \theta}\frac{dr_1 dr_2}{\{1+r_1^2+r_2^2\}^{2\al}}  \\
& &\hspace{4cm}\gga^{H_0,n}_{t}(\rho) \gga^{H_0,n}_{t} \Big( \sqrt{\rho^2 \sin^2\theta \big(1-\frac{r_1}{\rho \sin \theta}\big)^2+\rho^2 \cos^2\theta\big(1-\frac{r_2}{\rho \cos \theta}\big)^2}\Big) \ ,
\end{eqnarray*}
where, for technical reasons (in subsequent arguments), we have picked $H_1' \geq H_1$ and $H_2' \geq H_2$ such that $\frac34 <H_0+H_1'+H_2' \leq 1$.
At this point, observe that for all $\theta\in (\frac{\pi}{8},\frac{\pi}{4})$, $\rho\in (2,2^{n-1})$ and $r_1,r_2\in (0,\frac12 \rho \sin\theta)$, we have
$$\rho \geq \rho_{r,\theta}\triangleq \sqrt{\rho^2 \sin^2\theta \big(1-\frac{r_1}{\rho \sin \theta}\big)^2+\rho^2 \cos^2\theta\big(1-\frac{r_2}{\rho \cos \theta}\big)^2} \geq \frac12 \rho \geq 1 \ .$$
We are therefore in a position to apply the (forthcoming) lower bound (\ref{lower-bound-gga-n}), which entails, with the notation of Lemma \ref{tech-lem-explos},
\begin{eqnarray}
\lefteqn{\mathbb{E}\Big[ \|\<Psi2>^n(t,.)\|_{\mathcal{W}^{-2\al,2}(D)}^2\Big] \ \gtrsim \ \int_{\frac{\pi}{8}}^{\frac{\pi}{4}} d\theta \int_2^{2^{n-1}} \frac{d\rho}{\rho^{4H'_1+4H'_2-5}} }\nonumber\\
& &\hspace{2cm}  \int_0^{\frac12 \rho \sin \theta }\int_0^{\frac12 \rho \sin \theta}\frac{dr_1 dr_2}{\{1+r_1^2+r_2^2\}^{2\al}} \bigg[\frac{c\, t}{\rho^{1+2H_0}} +Q^{H_0}_t(\rho) \bigg] \bigg[\frac{c\, t}{\rho^{1+2H_0}} +Q^{H_0}_t(\rho_{r,\theta}) \bigg] \nonumber\\
&\gtrsim& t^2\bigg(\int_0^{\sin \frac{\pi}{8} }\int_0^{\sin \frac{\pi}{8}}\frac{dr_1 dr_2}{\{1+r_1^2+r_2^2\}^{2\al}}\bigg)\bigg(\int_2^{2^{n-1}} \frac{d\rho}{\rho^{4(H_0+H'_1+H'_2)-3}}\bigg)  +R^n_{t} \ ,\label{low-bou-pr}
\end{eqnarray}
where we have set
\begin{align*}
&R^n_{t}\triangleq \int_{\frac{\pi}{8}}^{\frac{\pi}{4}} d\theta \int_2^{2^{n-1}} \frac{d\rho}{\rho^{4H'_1+4H'_2-5}}\\
&\hspace{1cm}\int_0^{\frac12 \rho \sin \theta }\int_0^{\frac12 \rho \sin \theta}\frac{dr_1 dr_2}{\{1+r_1^2+r_2^2\}^{2\al}} \bigg[\frac{c\, t}{\rho^{1+2H_0}}Q_{H_0}(\rho_{r,\theta}) +\frac{c\, t}{\rho^{1+2H_0}}Q_{H_0}(\rho)+Q_{H_0}(\rho)Q_{H_0}(\rho_{r,\theta}) \bigg]  \ .
\end{align*}
Let us now show that $|R^n_t|$ is uniformly bounded with respect to $n$. In fact, thanks to (\ref{bou-q}), we can assert that for any $\varepsilon >0$, 
$$\sup_{t\in [0,1]}\bigg|\frac{c}{\rho^{1+2H_0}}Q^{H_0}_t(\rho_{r,\theta}) +\frac{c}{\rho^{1+2H_0}}Q^{H_0}_t(\rho)+Q^{H_0}_t(\rho)Q^{H_0}_t(\rho_{r,\theta}) \bigg| \lesssim \frac{1}{\rho^{3+4H_0-\varepsilon}}  \ .$$
Therefore,
\begin{eqnarray*}
\sup_{t\in [0,1]}\big| R^n_t \big|
&\lesssim &\bigg(\int_0^{\infty}\int_0^{\infty }\frac{dr_1 dr_2}{\{1+r_1^2+r_2^2\}^{2\al}} \bigg) \bigg(\int_2^\infty \frac{d\rho}{\rho^{4(H_0+H_1'+H_2')-2-\varepsilon}}  \bigg)\\
&\lesssim &\bigg(\int_0^{\infty}dr\, \frac{r}{\{1+r^2\}^{2\al}} \bigg) \bigg(\int_2^\infty \frac{d\rho}{\rho^{4(H_0+H_1'+H_2')-2-\varepsilon}}  \bigg) \ ,
\end{eqnarray*}
and, provided $\varepsilon >0$ is picked small enough, these two integrals are obviously finite, due to $\al > \frac12$ and $H_0+H_1'+H_2' > \frac34$. 

\smallskip

\noindent
Going back to (\ref{low-bou-pr}), we get the conclusion since, as $H_0+H_1'+H_2' \leq 1$,
$$\int_2^{2^{n-1}} \frac{d\rho}{\rho^{4(H_0+H'_1+H'_2)-3}}  \ \stackrel{n\to \infty}{\longrightarrow} \ \infty \ .$$

\

\begin{lemma}\label{tech-lem-explos}
There exists a constant $c>0$ such that for all ${H_0}\in (0,1)$, $\varepsilon >0$, $n\geq 1$, $t\in [0,1]$ and $\rho \in (1,2^n)$,
\begin{equation}\label{lower-bound-gga-n}
\gga^{H_0,n}_t(\rho) \geq \frac{c\, t}{\rho^{1+2H_0}} +Q^{H_0}_t(\rho) \geq 0 \ ,
\end{equation} 
with $Q^{H_0}$ satisfying 
\begin{equation}\label{bou-q}
\sup_{t\in [0,1]} |Q^{H_0}_t(\rho)|\leq \frac{c_{\varepsilon,{H_0}}}{ \rho^{2+2{H_0}-\varepsilon}} \ .
\end{equation}
\end{lemma}

\begin{proof}
We will lean on similar estimates as those of the proof of \cite[Proposition 2.4]{deya-wave}. Let us first recall the explicit expression (see the latter reference) $|\ga_t(\xi,\rho)|^2=c\big\{\Lambda_t(\xi,\rho)+\Lambda_t(-\xi,\rho)\}$, with $c>0$ and
$$
\Lambda_t(\xi,\rho)\triangleq \frac{1-\cos(t(\xi-\rho))}{\rho^2(\xi-\rho)^2} +\frac{\cos(t\rho)\{\cos(t\xi)-\cos(t\rho)\}}{\rho^2(\xi-\rho)(\xi+\rho)} \ .
$$
Thus, one has, for any $\rho \in (1,2^n)$,
$$\gga^{H_0,n}_t(\rho) \geq  \int_{-\rho}^{\rho} d\xi \, \frac{|\ga_t(\xi,\rho)|^2}{|\xi|^{2{H_0}-1}}=2c\int_{-\rho}^{\rho} d\xi \, \frac{\Lambda_t(\xi,\rho)}{|\xi|^{2{H_0}-1}} \geq 0 \ .$$
Decompose $\Lambda_t(\xi,\rho) \mathbf{1}_{\{-\rho <\xi<\rho\}}$ into $\Lambda_t(\xi,\rho)=\Lambda^1_t(\xi,\rho)+\Lambda^2_t(\xi,\rho)$, with 
$$\Lambda^1_t(\xi,\rho)\triangleq \frac{1-\cos(t(\xi-\rho))}{\rho^2(\xi-\rho)^2} \mathbf{1}_{\{\xi \geq \frac{\rho}{2}\}}$$
and
$$\Lambda^2_t(\xi,\rho)\triangleq \frac{1-\cos(t(\xi-\rho))}{\rho^2(\xi-\rho)^2} \mathbf{1}_{\{-\rho\leq \xi \leq \frac{\rho}{2}\}}+\frac{\cos(t\rho)\{\cos(t\xi)-\cos(t\rho)\}}{\rho^2(\xi-\rho)(\xi+\rho)}\mathbf{1}_{\{-\rho <\xi<\rho\}} \ .$$
On the one hand, it is easy to check that for all $\rho >1$, $\xi\in (-1,1)$ and $\varepsilon >0$,
$$\big|\Lambda^2_t(\rho\xi,\rho)\big| \lesssim \frac{1}{\rho^{4-\varepsilon}} \bigg[ 1+\frac{1}{|1-|\xi||^{1-\varepsilon}} \bigg] \ ,$$
and so
$$\bigg|\int_{-\rho}^{\rho} d\xi \, \frac{\Lambda^2_t(\xi,\rho)}{|\xi|^{2{H_0}-1}} \bigg| =\frac{1}{\rho^{2{H_0}-2}}\bigg|\int_{-1}^{1} d\xi \, \frac{\Lambda^2_t(\rho \xi,\rho)}{|\xi|^{2{H_0}-1}} \bigg| \lesssim \frac{1}{\rho^{2{H_0}+2-\varepsilon}} \ .$$
On the other hand,
\begin{eqnarray*}
\lefteqn{\int_{-\rho}^{\rho} d\xi \, \frac{\Lambda^1_t(\xi,\rho)}{|\xi|^{2{H_0}-1}}  \ =\ \frac{1}{\rho^{2{H_0}+2}} \int_{\frac12}^1 \frac{d\xi}{|\xi|^{2{H_0}-1}} \, \frac{1-\cos(t\rho (1-\xi))}{(1-\xi)^2}}\\
&=&\frac{t}{\rho^{2{H_0}+1}} \int_{0}^{\frac{t\rho}{2}} \frac{d\xi}{|1-\frac{\xi}{t\rho}|^{2{H_0}-1}} \, \frac{1-\cos\xi}{\xi^2}\\
&=&\frac{t}{\rho^{2{H_0}+1}} \int_{0}^{\infty} d\xi \, \frac{1-\cos\xi}{\xi^2}+\frac{t}{\rho^{2{H_0}+1}} \bigg[\int_{0}^{\frac{t\rho}{2}} \frac{d\xi}{|1-\frac{\xi}{t\rho}|^{2{H_0}-1}} \, \frac{1-\cos\xi}{\xi^2}-\int_{0}^{\infty} d\xi \, \frac{1-\cos\xi}{\xi^2} \bigg] \ .
\end{eqnarray*}
The conclusion now follows immediately from the result of \cite[Lemma 2.5]{deya-wave}.
\end{proof}

\

\


\begin{thebibliography}{99}



\bibitem{deya-wave}
A. Deya: A non-linear wave equation with fractional perturbation. {\it Arxiv preprint} (2017).

\smallskip
 
\bibitem{ginibre-velo}
J. Ginibre and G. Velo: Generalized Strichartz inequalities for the wave equation, {\it J. Funct. Anal.} {\bf 133} (1995), 50-68.

\smallskip

\bibitem{gubinelli-koch-oh}
M. Gubinelli, H. Koch and T. Oh: Renormalization of the two-dimensional stochastic non linear wave equations. {\it Arxiv preprint} (2017).

\smallskip

\bibitem{hairer-labbe}
M. Hairer and C. Labb{\'e}: Multiplicative stochastic heat equations on the whole space. To appear in {\it J. Eur. Math. Soc}.

\smallskip

\bibitem{hairer-pardoux}
M. Hairer and E. Pardoux: A Wong-Zakai theorem for stochastic PDEs. {\it Jour. Math. Soc. Japan} {\bf 67} (2015), no. 4, 1551--1604.



\smallskip

\bibitem{mourrat-weber}
J.-C. Mourrat and H. Weber: The dynamic $\varphi_3^4$ model comes down from infinity. To appear in {\it Comm. Math. Phys.}.

\smallskip

\bibitem{runst-sickel}
T. Runst and W. Sickel: Sobolev Spaces of Fractional Order, Nemytskij Operators, and Nonlinear Partial Differential Equations. de Gruyter Series in Nonlinear Analysis and Applications 3, Berlin (1996).


\end{thebibliography}
\end{document}